\pdfoutput=1
\documentclass[12pt, a4paper]{article}
\usepackage[latin1]{inputenc}
\usepackage[english]{babel}
\usepackage{amsmath,verbatim,amsthm,mathrsfs,stmaryrd}
\usepackage[english]{babel}
\usepackage{indentfirst}
\usepackage{amstext}
\usepackage{enumerate}
\usepackage[all]{xy}
\usepackage{amsfonts} 
\usepackage{textcomp}
\usepackage{amssymb}
\usepackage{graphicx}
\usepackage{setspace}
\usepackage{color}
\usepackage{ulem}
\usepackage[dvipsnames]{xcolor}
\usepackage{graphicx}
\graphicspath{ {./images/} }
\usepackage[all]{xy}
\usepackage[noabbrev, capitalise, nameinlink, nosort]{cleveref}
\usepackage{tikz}
\usetikzlibrary{patterns,matrix,arrows}
\usetikzlibrary{arrows}

\newcommand {\N}{\mathbb{N}}

\newcommand {\p}{\mathfrak{p}}

\newcommand{\bgln}{\begin{eqnarray}} 
\newcommand{\egln}{\end{eqnarray}}
\newcommand{\bgl}{\begin{equation}} 
\newcommand{\egl}{\end{equation}}

\usepackage{geometry}
\newtheorem{teorema}{theorem}[section]
\newtheorem{lemma}[teorema]{Lemma}
\newtheorem{corollary}[teorema]{Corollary}
\newtheorem{definition}[teorema]{Definition}
\newtheorem{proposition}[teorema]{Proposition}

\newtheorem{example}[teorema]{Example}

\theoremstyle{remark}
\theoremstyle{definition}
\newtheorem{remark}[teorema]{Remark}
\newtheorem{theorem}[teorema]{Theorem}

\begin{document}

\author{Daniel Gon\c{c}alves\footnote{Partially supported by Capes-PrInt and Conselho Nacional de Desenvolvimento Cient\'ifico e Tecnol\'ogico (CNPq) and Capes-PrInt - Brazil.}\hspace{0.3pc} and Bruno B. Uggioni}

\title{Shadowing for local homeomorphisms, with applications to edge shift spaces of infinite graphs}
\maketitle

\begin{abstract} 

In this paper, we develop the basic theory of the shadowing property for local homeomorphisms of metric locally compact spaces, with a focus on applications to edge shift spaces connected with C*-algebra theory. For the local homeomorphism (the Deaconu-Renault system) associated with a direct graph, we completely characterize the shadowing property in terms of conditions on sets of paths. Using these results, we single out classes of graphs for which the associated system presents the shadowing property, fully characterize the shadowing property for systems associated with certain graphs, and show that the system associated with the rose of infinite petals presents the shadowing property and that the Renewal shift system does not present the shadowing property.

\vspace{5mm}


\vspace{5mm}

\end{abstract}

\section{Introduction}

There is no standard approach to shift spaces over infinite alphabets. In the case of countable state Markov shifts, i.e. shift spaces associated with an infinite matrix, the more common approach is to use a direct generalization of the finite alphabet case, that is, to consider the shift space as the product space of the alphabet with the pro-discrete topology, see \cite{Tullio, Kitchens, Sarig, Sarig1} for example. One of the difficulties with this approach is that the shift space obtained in this way is not even locally compact. Motivated by this, various approaches to shift spaces have appeared, see \cite{OTW} 
 for an overview, and recently researchers have looked into shift spaces arising in connection with C*-algebra theory. In the latter context, the shift spaces also include finite sequences and the topology differs from the product topology, see \cite{subshift, EL, GRultrapartial, GTasca, OTW} for details. Further to the connection with C*-algebra theory, the interest in these shift spaces arises as researchers realize that new dynamical phenomena present themselves in the part of the space formed by the finite sequences (see for example \cite{BisEx}, where thermodynamic formalism of the aforementioned generalized countable Markov shifts is studied) or in relation to the different topology (see for example \cite{GR44, GSo, GSoStar}). 

Our initial intention, when we began developing this paper, was to study the shadowing property of edge shift spaces of graphs that arise in connection with C*-algebra theory and verify if new phenomena also emerge in this context. While searching the literature for infinite alphabet shift spaces associated with C*-algebras, we observed that all approaches (see \cite{subshift, EL, GRultrapartial, GTasca, OTW}) incorporate a local homeomorphism defined in the space, known as a Deaconu-Renault system. This local homeomorphism is either considered as the shift or its domain is a dense subset of the shift space. Recently, the dynamical properties of Deaconu-Renault systems have garnered significant attention (see, for example, \cite{ABCE, CRST, GRT}). Given this, we decided to develop the fundamental theory of shadowing in the context of metric Deaconu-Renault systems and apply it to infinite graphs. In comparison to the conventional product topology approach, we discovered a broader range of phenomena (further details provided below), which arise not from the existence of finite sequences in the shift space but due to the employed metric/topology.

Shadowing is a well-established concept in dynamical systems. It has been shown to be a relatively common property in the space of diverse dynamical systems (see \cite{medd} and \cite{opro}) and it is associated with the stability of such systems (\cite{walters1}) and to other dynamical and metric properties. Considering shifts over finite alphabets, the classical result of Walters \cite{walters1} says that a shift has the shadowing property if, and only if, it is a shift of finite type. In the infinite alphabet case, with the product topology, a similar result holds, i.e. a shift has the shadowing property if and only if it is a shift of finite order (a generalization of a shift of finite type), see \cite{DGM, MR}. In the context of the Deaconu-Renault system associated with a graph, we provide an example of a graph with only two vertices in which the associated system does not possess the shadowing property. This raises an immediate question: finding a description of graphs for which the associated system exhibits the shadowing property. In our work, we present a general characterization of the shadowing property and utilize it to describe classes of graphs with and without the shadowing property.

Returning to the general setting of a local homeomorphism in a metrically locally compact set, it is worth noting that the complete theory of shadowing needs to be developed. For instance, in \cite{Pilyugin}, it is proven that finite shadowing and infinite shadowing are equivalent for compact spaces. However, this equivalence does not hold in the general theory of shadowing for non-compact spaces (see \cite{DGM}), nor does it hold for Deaconu-Renault systems. Additionally, other properties of the shadowing property need to be proven or verified. In this paper, we will focus on the ones required to describe the shadowing of the Deaconu-Renault system associated with an infinite graph. Specifically, the paper is organized as follows:

In Section~\ref{Slh} we recall the definition of a Deaconu-Renault system, define the (finite) shadowing property for such systems, study the relationship between the shadowing property and the finite shadowing property, show that the (finite) shadowing property is preserved under uniformly equivalent metrics, and show that the shadowing property may be studied via a restriction of the system to a dense subset. In Section~\ref{frentefria}, we recall the metric Deaconu-Renault system associated with a graph and show properties of the metric that will be necessary for our work. Since the metric is non-trivial, in Section~\ref{capivara} we characterize the finite shadowing property and the shadowing property in terms of conditions on sets of paths, which we call the Finite Path Condition (see Theorem~\ref{teofiniteshadowing}) and the First and Second Infinite Path Conditions (see Theorem~\ref{teoinfiniteshadowing}), respectively. Finally, in Section~\ref{lagarta}, we use the results of Section~\ref{capivara} to single out classes of graphs for which the associated Deaconu-Renault system presents the (finite) shadowing property. Among the examples, we show that the system associated with the rose with infinite petals, or with a wandering graph, has the shadowing property. For the class of graphs which contain what we call a finite attractor subgraph we completely characterize the shadowing property, see Proposition~\ref{propproAshadowing}, and using this result we show that the system associated with the Renewal shift does not have the shadowing property. The last two examples of the paper concern a graph with only two vertices such that the associated system does not present the finite shadowing property and a graph such that the associated system presents the finite shadowing property but not the shadowing property.

\section{Shadowing for local homeomorphisms}\label{Slh}

Throughout this paper, we consider the set of the natural numbers as the set $\{1,2, \ldots\}$. 

Deaconu-Renault systems, or singly generated dynamical systems, were introduced independently by Deaconu in \cite{D95} and by Renault in \cite{Re00}. In this subsection we recall some of its concepts and terminology, as introduced in \cite{ABCE}, which in turn follows \cite{CRST}. We also introduce the notion of shadowing for Deaconu-Renault systems and study some of its properties, such as the relations between the shadowing property and the finite shadowing property, the equivalence of shadowing via uniformly equivalent metrics, and the effect of restrictions on the shadowing property.

\begin{definition} \cite{ABCE} A \textit{Deaconu-Renault system} is a pair $(X,\sigma)$ consisting of a locally compact Hausdorff space $X$, and a local homeomorphism $\sigma: Dom(\sigma)\longrightarrow Im(\sigma)$ from an open set $Dom (\sigma)\subseteq X$ to an open set $Im(\sigma)\subseteq X$. Inductively define $Dom(\sigma^n):=\sigma^{-1}(Dom(\sigma^{n-1}))$, so each $\sigma^n:Dom(\sigma^n)\longrightarrow Im(\sigma^n)$ is a local homeomorphism and $\sigma^m\circ \sigma^n=\sigma^{m+n}$ on $Dom(\sigma^{m+n})$.
\end{definition}

The definition of the shadowing property can be readily adapted to the context of metric Deaconu-Renault systems. We do this below.

\begin{definition} Let $(X,\sigma)$ be a Deaconu-Renault system, with $X$ a metric space.  Given $\delta>0$, a (finite) \textit{$\delta-chain$} in $(X,\sigma)$ is a (finite) sequence $(x_n)$ such that $d(\sigma(x_n),x_{n+1})< \delta$ for all $n$, where it is implicity that $x_n\in Dom(\sigma)$ when $\sigma(x_n)$ is written. 

We say that a point $x\in X$ \textit{$\varepsilon-$shadows} a (finite) sequence $(y_i) $ in $X$ if $d(\sigma^i(x), y_i)<\varepsilon$ for all $i$. Notice that this implies that $x\in Dom(\sigma^i)$ for all $i$.
\end{definition}

\begin{remark} Notice that if $(x_i)_{i=1}^k$ is a finite $\delta$-chain then $x_k$ does not, necessarily, belongs to $Dom(\sigma)$.
\end{remark}

With the above setup, we define the shadowing property in the same manner as the standard definition, see below.

\begin{definition}
We say that a Deaconu-Renault system has the (finite) shadowing property if for every $\varepsilon>0$, there exists a $\delta>0$ such that every (finite) $\delta-$chain is $\varepsilon-$shadowed by a point $x\in X$.
\end{definition}

\begin{remark}
    We emphasize that when we do not specify the term ``finite" in front of ``shadowing property" we are referring to the shadowing property via infinite sequences. Moreover, when $Dom(\sigma)=X$, the above definitions coincide with the usual definitions of the shadowing property and the finite shadowing property. 
\end{remark}

Some of our follow-up results deal with Deaconu-Renault systems $(X,\sigma)$ such that $Dom(\sigma)$ is dense in $X$. This condition can be characterized as follows. 

\begin{proposition}
\label{niceone} Let $(X,\sigma)$ be a Deaconu-Renault system. Then, $Dom(\sigma)$ is dense in $X$ if, and only if,  for every $n\in \{0,1,2,\ldots\}$, $x\in Dom (\sigma^n)$, and $V$ an open neighborhood of $x$, there is a $y\in V\cap Dom(\sigma^{n+1})$. 
\end{proposition}
\begin{proof}

Suppose that $x\in Dom(\sigma^n)$ and let $V$ be an open set that contains $x$. Since $\sigma$ is a local homeomorphism, there is an open set $U\subseteq V$ such that $\sigma^n: U\rightarrow \sigma^n (U)$ is a homeomorphism. Notice that $\sigma^n (U)$ is a neighboorhood of $\sigma^n(x)$. So, by hypothesis, there exists a $z\in \sigma^n(U)\cap Dom(\sigma)$. But then $\sigma|_U^{-n}(z) \in Dom(\sigma^{n+1}) \cap V$ as desired.

\end{proof}

Next, we focus on describing the relations between the shadowing property and the finite shadowing property. We begin showing that, for Deaconu-Renault systems with dense domains, the shadowing property implies the finite shadowing property.

\begin{proposition}\label{jaburu}
Let $(X,\sigma)$ be a Deaconu-Renault system, with $X$ metric and such that $Dom(\sigma)$ is dense in $X$. Then, the shadowing property implies the finite shadowing property.
\end{proposition}
\begin{proof}
Suppose that $(X,\sigma)$ has the shadowing property. Given $\varepsilon>0$, let $\delta<\frac {\varepsilon}{2}$ be such that any $\delta-$chain is $\frac{\varepsilon}{2}-$shadowed.

Let $(x_n)_{n=1}^{k}$ be a finite $\frac{\delta}{2}-$chain. Since $Dom(\sigma)$ is dense in $X$, there exists $y_k\in B(\sigma(x_{k-1}),\frac{\delta}{2})$ such that $y_k\in Dom(\sigma)$. Notice that $d(y_k,x_k)<\delta$ (using the triangle inequality with $\sigma(x_{k-1}$)). Proceeding inductively, define $y_{k+j}$ to be a point in $B(\sigma(y_{k+j-1}),\delta)\cap Dom(\sigma)$ (which exists since $Dom(\sigma)$ is dense in $X$), for $j=1,2,\ldots$. Notice that $x=x_1, \ldots, x_{k-1}, y_k, y_{k+1},\ldots$ is a $\delta-$chain and hence, by the choice of $\delta$, there is $z\in X$ that $ \frac{\varepsilon}{2}-$ shadows $x$. In particular, $z$ $\varepsilon$-shadows $(x_n)_{n=1}^{k}$ (notice that $d(\sigma^k(z),x_k)\leq d(\sigma^k(z),y_k) + d(y_k,x_k)< \varepsilon$).

\end{proof}

The converse of the above proposition does not hold in general, see Example~\ref{exempfinitevinfinite}. Below we give a condition under which it does hold.

\begin{definition}
 Let $(X,\sigma)$ be a Deaconu-Renault system with $X$ metric. We say that $(X,\sigma)$ is uniformly locally compact if there is $\varepsilon' > 0$ such that for all $x \in Dom(\sigma^n) $, the closure of  $B(x, \varepsilon')$, which is denoted by $\overline{B}(x,\varepsilon')$, is a compact set.
\end{definition}

The following result is a generalization of \cite[Proposition~2.3.2]{DGM}, with a shorter proof.
\begin{proposition}
Let $(X,\sigma)$ be a uniformly locally compact Deaconu-Renault system, with $X$ metric and such that $Dom(\sigma)$ is dense in $X$. Then, $(X,\sigma)$ has the shadowing property if, and only if, it has the finite shadowing property. 
\end{proposition}
\begin{proof}

That the shadowing property implies the finite shadowing property follows from Proposition~\ref{jaburu}.

For the converse, suppose that $(X,\sigma)$ has the finite shadowing property. Let $\varepsilon > 0$. We assume, without loss of generality, that $\varepsilon$ is sufficiently small so that, for each $x \in Dom(\sigma)$, the closure of the open ball $B(x,\varepsilon)$ is compact. Let $\delta>0$ be such that every finite $\delta$-chain is $\frac{\varepsilon}{2}$-shadowed by some point. 

Given $(x_{i})_{i=1}^{\infty}$ an infinite $\delta$-chain we consider the collection of closed subsets of $\overline{B}(x_1,\varepsilon)$ given by $$A_i:=\sigma^{-i+1}\left(\overline{B}(x_i,\frac{\varepsilon}{2})\right) \cap \overline{B}(x_1,\varepsilon), \ i=1,2\ldots.$$
We show that the above family has the finite intersection property. Indeed, given a finite sub-collection of sets, say $\{A_j\}_{j\in J}$, let  $k:=\max\{j:j\in J\}$. Notice that $(x_{i})_{i=1}^{k}$ is a finite $\delta$-chain.
Hence, there exists $z_k \in Dom(\sigma^k)$ that $\frac{\varepsilon}{2}$-shadows $(x_{i})_{i=1}^{k}$, which implies that $z_k\in \cap_{j\in J} A_j$, as desired. We conclude, by compactness, that there exists $z\in \cap_{i=1}^{\infty} A_i$, which is a point that $\varepsilon$-shadows $(x_{i})_{i=1}^{\infty}$.

\end{proof}

\begin{remark}
The proof above that the finite shadowing property implies the shadowing property holds without the assumption that $Dom(\sigma)$ is dense.
\end{remark}

Our next goal in this section is to show that the (finite) shadowing property is preserved by uniformly equivalent metrics. This will be important when studying the shadowing property of Deaconu-Renault systems associated with infinite graphs. The relevant definition and result follow below.

\begin{definition}
Let $d_1$ and $d_2$ be two metrics on $X$. We say that $d_2$ is uniformly continuous with respect to $d_1$ if $Id_X:(X,d_1)\rightarrow (X,d_2)$ is uniformly continuous. Moreover, we say that $d_1$ is uniformly equivalent to $d_2$ if $d_1$ is uniformly continuous with respect to $d_2$ and $d_2$ is uniformly continuous with respect to $d_1$.
\end{definition}

\begin{proposition}\label{mequivalent}
Let $(X, \sigma, d_{1})$ and $(X, \sigma, d_{2})$ be Deaconu-Renault systems such that the metrics are uniformly equivalent. Then, $(X, \sigma, d_{1})$ has the (finite) shadowing property if, and only if, $(X, \sigma, d_{2})$ does. 
\end{proposition}
\begin{proof}
Suppose that $(X, \sigma, d_{1})$ has the shadowing property and fix $\varepsilon > 0$. By the metric equivalence, there exists $\delta_{1}>0$ such that $d_{1}(x,y) < \delta_{1} \Rightarrow d_{2}(x,y) < \varepsilon$ for all $x,y \in X$. By the shadowing property, there exists $\delta'>0$ such that all $(\delta',d_{1})$ - chain is $(\delta_{1},d_{1})$ - shadowed by a point of $X$. Again, by the metric equivalence, we find a $\delta_{2} > 0$ such that $d_{2}(x,y) < \delta_{2} \Rightarrow d_{1}(x,y) < \delta'$ for all $x,y \in X$.

Let $\delta = \delta_{2}$ and $(x_{n})$ be a $(\delta, d_{2})$ - chain. Then, by our construction, $(x_{n})$ is also a $(\delta',d_{1})$ - chain, which is $(\delta_{1},d_{1})$ - shadowed by a point $z \in X$, which also $(\varepsilon, d_{2})$ - shadows $(x_{n})$. Hence, $(X, \sigma, d_{2})$ has the shadowing property. The converse of the proposition is proved following the same steps.
\end{proof}

When dealing with the shadowing property for a Deaconu-Renault system $(X,\sigma)$, in particular for a Deaconu-Renault system associated with a graph, it is convenient to check the shadowing property only in a restriction of $X$. By this, we mean that if $D\subset X$ then $(D,\sigma)$ has the (finite) shadowing property if for every $\varepsilon$, there exists a $\delta>0$ such that every $\delta$-chain formed by elements in $D$ is $\varepsilon$-shadowed by an element of $D$.

Next we show that, for Deaconu-Renault systems, it is enough to verify shadowing in a dense subset. This is analogous to \cite[Theorem~2.1.7]{DGM}, but we do not require uniform continuity of the map.

\begin{proposition}\label{denseshad}
Let $(X,\sigma)$ be a Deaconu-Renault system and $D$ be a dense subset of $Dom(\sigma)$. If $(D,\sigma)$ has the shadowing property, then $(X,\sigma)$ has the shadowing property. If, moreover, $Dom(\sigma)$ is dense in $X$ and $(D,\sigma)$ has the finite shadowing property, then $(X,\sigma)$ has the finite shadowing property.
\end{proposition}
\begin{proof}

Suppose that $D$ is a dense subset of $Dom(\sigma)$ such that $(D,\sigma)$ has the shadowing property.
Given $\varepsilon>0$, let $\delta<\varepsilon$ be such that every $\delta$-chain in $D$ is $\frac{\varepsilon}{2}$ shadowed in $D$.

Let $(x_i)$ be a $\frac{\delta}{3}$-chain in $X$. Define a sequence $(y_i)$ in the following way: If $x_i\in D$, then let $y_i=x_i$. If $x_i\notin D$, then let $y_i\in D$ be such that $d(x_i,y_i)<\frac{\delta}{3}$ and $d(\sigma(x_i),\sigma(y_i))<\frac{\delta}{3}$ (such $y_i$ exists by the continuity of $\sigma$ in $x_i$). Notice that $(y_i)$ is a $\delta$-chain. Indeed, \[d(y_{n+1},\sigma(y_n))\leq d(y_{n+1},x_{n+1})+d(x_{n+1},\sigma(x_n))+d(\sigma(x_n),\sigma(y_n))\leq \frac{\delta}{3}+\frac{\delta}{3}+\frac{\delta}{3} = \delta.\]
So, $(y_i)$ is $\frac{\varepsilon}{2}$-shadowed by a point, say $z$, in $D$. Then, $z$ also $\varepsilon$-shadows $(x_i)$. Indeed, \[d(f^k(z),x_k)\leq d(f^k(z),y_k)+d(y_k,x_k)\leq \frac{\varepsilon}{2}+\frac{\varepsilon}{3} <\varepsilon.\]

The proof regarding the finite shadowing property, further assuming that $Dom(\sigma)$ is dense in $X$, is analogous. The only difference is that given $(x_i)_{i=1}^k$ a $\frac{\delta}{3}$-chain in $X$, if $x_k$ is not in $D$, then one should choose $y_k\in D$ such that $d(x_k,y_k)<\frac{\delta}{3}$ (such $y_k$ exists since $D$ is dense in $Dom(\sigma)$, which in turn is dense in $X$). The remainder of the proof follows exactly the shadowing property case.

\end{proof}

\begin{remark}
The converse of Proposition~\ref{denseshad} does not hold in general. For example, consider the graph with only one vertex and infinite edges, as in Corollary \ref{coroalledges}. There, we show that the associated Deaconu-Renault system presents the shadowing property, but the finite sequences are dense and do not have the shadowing property (since any $\delta$-chain, with small enough $\delta$, must be shadowed by an infinite sequence).
\end{remark}

Although the converse of Proposition~\ref{denseshad} does not hold in general, it does hold for certain Deaconu-Renault systems with an appropriate choice of a dense subset, as we see below.

\begin{corollary}\label{chega}
Let $(X,\sigma)$ be a Deaconu-Renault system, and suppose that $$D:=\bigcap_{n\in \N} Dom(\sigma^n)$$ is dense in $Dom(\sigma)$. Then, $(X,\sigma)$ has the shadowing property if, and only if, $(D,\sigma)$ has the shadowing property. If, moreover, $Dom(\sigma)$ is dense in $X$, then $(X,\sigma)$ has the finite shadowing property if, and only if, $(D,\sigma)$ has the finite shadowing property.
\end{corollary}
\begin{proof}
The first part of the result follows from Proposition~\ref{denseshad} and the observation that if $(X,\sigma)$ has the shadowing property, then any $\delta-$chain in $D$ must be shadowed by an element in $D$ (so that the iterates of $\sigma$ are all defined).

For the second part of the corollary, assume that $Dom(\sigma)$ is dense in $X$. By Proposition~\ref{denseshad}, if $(D,\sigma)$ has the finite shadowing property then $(X,\sigma)$ has the finite shadowing property. Suppose that $(X,\sigma)$ has the finite shadowing property, let $\varepsilon>0$ and choose $\delta>0$ that witness $\varepsilon$-shadowing. Let $(x_i)_{i=1}^k$ be a finite $\delta$-chain in $D$. By hypothesis, there exists $z$ that $\varepsilon$-shadows $(x_i)_{i=1}^k$, that is, $z\in \bigcap_{i=1}^k \sigma^{-i+1}\left( B(x_i,\varepsilon)\right)$. Since $D$ is dense in $Dom(\sigma)$ and $ \bigcap_{i=1}^k \sigma^{-i+1}\left( B(x_i,\varepsilon)\right)$ is open, there exists $y\in D \cap \bigcap_{i=1}^k \sigma^{-i+1}\left( B(x_i,\varepsilon)\right)$. Clearly such $y$ $\varepsilon$-shadows $(x_i)_{i=1}^k$.
\end{proof}

\section{The metric Deaconu-Renault system associated with an infinite graph.}\label{frentefria}

In this section, we recall the construction of the standard Deaconu-Renault system associated with a graph. This system serves as the basis for the construction of the groupoid C*-algebra associated with the graph, see \cite{Mike, KPRR, Paterson} for example.
For a finite graph, the mentioned Deaconu-Renault system coincides with the system formed by the usual edge shift space and the shift map, for which shadowing is well understood, see \cite{walters1}. This system has been generalized to ultragraphs and general subshifts, see \cite{subshift, GRultrapartial, GTasca} for example. Since the shadowing phenomena is very diverse on shifts over an infinite alphabet, we focus on the Deaconu-Renault system associated with an infinite graph. We start with some background material, followed by a study of the associated metric (which is fundamental for the characterization of the shadowing property we give in Section~\ref{capivara}).

\subsection{Background}

We start this section by making precise the concepts regarding graphs that we will use.

\begin{definition}\label{def of ultragraph}
A \textit{graph} is a quadruple $E=(E^0, E^1, r,s)$ consisting of two countable sets $E^0, E^1$, and maps $r, s:E^1 \to E^0$. Elements of $E^0$ are called vertices and elements of $E^1$ are called edges.
\end{definition}

Let $E$ be a graph. A \textit{finite path} in $E$ is either a $vertex$ or a sequence of edges $e_{1}\ldots e_{k}$ such that $r(e_i)=s(e_{i+1})$ for $1\leq i\leq k$. If we write $\alpha=e_{1}\ldots e_{k}$, then the length $\left|  \alpha\right|  $ of
$\alpha$ is $k$. The length of a vertex $v\in E^0$ is
zero. The set of all
 paths of length $n$ in $E$ is denoted by $\mathfrak{p}^{n}$ and the set of all paths of finite length by $\mathfrak{p}$. An \textit{infinite path} in $E$ is an infinite sequence of edges $\gamma=e_{1}e_{2}\ldots$ in $\prod E^{1}$, where
$s\left(  e_{i+1}\right)  = r\left(  e_{i}\right)  $ for all $i$. The set of
infinite paths in $E$ is denoted by $\mathfrak
{p}^{\infty}$. The length $\left|  \gamma\right|  $ of $\gamma\in\mathfrak
{p}^{\infty}$ is defined to be $\infty$. A vertex $v$ in $E^0$ is
called a \textit{sink} if $\left|  s^{-1}\left(  v\right)  \right|  =0$ and is
called an \textit{infinite emitter} if $\left|  s^{-1}\left(  v\right)  \right|
=\infty$.

We extend the range and maps to $E^0$, by defining $r(v)=s(v)=v$ for all $v\in E^0$, and extend the source map $s$ to $\mathfrak
{p}^{\infty}$, by defining $s(\gamma)=s\left(  e_{1}\right)  $, where
$\gamma=e_{1}e_{2}\ldots$. We may concatenate paths as follows. If $\alpha \in\mathfrak{p}\setminus \mathfrak{p}^0$ and $\gamma \in \mathfrak{p}^{\infty} \cup \mathfrak{p}\setminus \mathfrak{p}^0$ are such that $s\left(\gamma\right) =r\left(
\alpha\right)$, then the concatenation is defined as the path $\alpha\gamma$. Moreover, $s(\gamma) \gamma = \gamma = \gamma r(\gamma)$. If for paths $\beta$ and $\gamma$ there exists a finite path $\gamma'$ such that $\beta = \gamma \gamma'$ then we say that $\gamma$ is an initial segment of $\beta$. 

Since the edge shift of a graph is usually considered for graphs without sinks, we will restrict ourselves to graphs without sinks. We make this assumption explicit below:

{\bf Throughout assumption:} From now on all graphs in this paper are assumed to have no sinks. 

Next, we recall the Deaconu-Renault system associated with a graph. For a more detailed approach, see \cite{Mike, GRultrapartial, GTasca}.

Let \[X= \mathfrak{p}^{\infty} \cup X_{fin},\] where $X_{fin} = \{\alpha \in \mathfrak p \text{ such that } r(\alpha) \text{ is an infinite emitter}\}. $  The topology we consider on $X$ has a basis given by the collection $$\{D_{\beta}: \beta \in \mathfrak{p}, |\beta|\geq 1\ \} \cup \{D_{\beta,F}:\beta \in X_{fin}, F\subset s^{-1}(r(\beta)), |F|<\infty \},$$ where for each $\beta\in \mathfrak{p}$ we have that $$D_{\beta}= \{y \in X: y = \beta \gamma', s(\gamma')=r(\beta)\},$$ and, for $\beta\in X_{fin}$ and $F$ a finite subset of $\varepsilon\left( B \right)$,  $$D_{\beta,F}=  \{\beta\}\cup\{y \in X: y = \beta \gamma', \gamma_1' \in s^{-1}(r(\beta))\setminus F\}.$$
 
\begin{remark}\label{cylindersets} For every $\beta\in \mathfrak{p}$, we identify $D_{\beta}$ with $D_{\beta,F}$, where $F=\emptyset$. Furthermore, we call the basic elements of the topology of $X$ given above by \textit{cylinder sets}.
\end{remark}

Next, we define the Deaconu-Renault system associated with a graph. 

\begin{definition} 
 For $n\in\{0,1,2,\ldots\}$, define the following subsets of $X$:
$$X^n=\p^n\cap X_{fin}=\{x\in X_{fin}:|x|=n\}; 
\ X^{\geq n}=\displaystyle\bigcup_{k\geq n} X^k\text{ and }
 X_{\infty}^{\geq n}=X^{\geq n}\cup \p^\infty .$$
	\end{definition}

Notice that $X_{\infty}^{\geq 1}=X\backslash X^0$ and that $X_{\infty}^{\geq 0}=X$.
	
	\begin{definition}\label{def_shift}
	Let $E$ be a graph and $X$ as above. We define the shift map $\sigma:X_{\infty}^{\geq 1}\rightarrow X$ by:
		\begin{center}
		$\sigma(x)=
		\left\{
		\begin{array}{ll}
		\gamma_2\gamma_3\ldots, & \text{if $x=\gamma_1\gamma_2\ldots \in \p^\infty$};\\
		\gamma_2\ldots\gamma_n, & \text{if $x=\gamma_1\ldots\gamma_n \in X^{\geq 2}$};\\
		r(x), & \text{if $x=\gamma_1 \in X^1$}.\\
				\end{array}
		\right.$  
		\end{center}
		For $n> 1$ we define $\sigma^n$ as the composition $n$ times of $\sigma$, and for $n=0$ we define $\sigma^0$ as the identity. When we write $\sigma^n(x)$ we are implicitly assuming that $x\in X_{\infty}^{\geq n}$. 
		
	\end{definition}

It is shown, see \cite[Proposition~5.4]{TascaDaniel} for example, that $(X,\sigma)$ is a Deaconu-Renault system and the infinite sequences are dense in $X$. Moreover, a family of metrics for the topology in $X$ is described in \cite{Webster}, see also \cite{BrunoDaniel, BD1}. We summarize this in the proposition below. 

\begin{proposition}\label{saguinasala}
Let $E$ be a graph. Then $(X,\sigma)$ is a Deaconu-Renault system, where $X$ is a metric space and $Dom(\sigma)=X_{\infty}^{\geq 1}$. Moreover, $\bigcap_{n\in \N} Dom(\sigma^n)=\p^\infty $ is dense in $Dom(\sigma)$.
\end{proposition}

Throughout our work, the Deaconu-Renault system associated with a graph $E$ will be the one given above. When we study shadowing for these systems, we will focus on the shadowing of infinite paths, thanks to the next proposition.

\begin{proposition}
Let $E$ be a graph, $(X,\sigma)$ the associated Deaconu-Renault system, and let $D$ be the set of all infinite paths in $X$. Then, $(X,\sigma)$ has the (finite) shadowing property if, and only if, $(D,\sigma)$ has the (finite) shadowing property
\end{proposition}
\begin{proof}
The result follows from Corollary~\ref{chega} and the observation that $  \bigcap_{n\in \N} Dom(\sigma^n)$ equals the set of infinite paths in $X$, which is dense in $X$ by Proposition~\ref{saguinasala}.
\end{proof}

\subsection{The metric on the Deaconu-Renault system associated with an infinite graph.}

In \cite{Webster} an explicit family of metrics for the topology on the space $X$ associated with a graph is described. Similar metrics have also been defined for Ott-Tomforde-Willis subshifts, \cite{OTW}, and for ultragraph shift spaces, \cite{BrunoDaniel}. The results regarding metrics on an ultragraph are also valid for a graph since the latter is a generalization of the former. In particular, in \cite{GRT}, it is shown that the metrics in the aforementioned family are uniformly equivalent. Hence, by Proposition~\ref{mequivalent}, we can pick any of these metrics to study the shadowing property. In fact, there is one metric, which first appeared in \cite[Remark~7.14]{GRT} and we call edge-ordered,  which has good properties for the study of shadowing. We review this metric below, followed by a  study of some of its properties.

List the set of finite paths as $\mathfrak{p} = \{ p_1, p_2, p_3, \ldots \}$ and define, for $x,y \in X$,
\begin{equation}\label{definmetricanova}
d_X (x,y) := \begin{cases} 1/2^i & \text{$i \in \N$ is the smallest value such that $p_i$ is an initial} \\  & \text{ \ \ \ segment of one of $x$ or $y$ but not the other,} \\
0 & \text{if $x=y$}.
\end{cases}
\end{equation} 

 Clearly, the metric $d_X$ depends on the order we choose for $\mathfrak{p} = \{ p_1, p_2, p_3, \ldots \}$. Nevertheless,  in \cite{GRT} it is shown that different orders give uniformly equivalent metrics. So, by Proposition~\ref{mequivalent}, we can choose any order. In our study of the shadowing property, we always use an enumeration of $\mathfrak p$ that satisfies the definition below.

\begin{definition}\label{edgeordered} Let $E$ be a graph. We say that an enumeration $\{p_0, p_1, p_2,\ldots \}$ of $\mathfrak p$ is \textit{edge ordered} if,  for every edge $e_i$, the source of $e_i$ and all paths of length less than $i$, formed by edges $e_k$ with $k<i$, appear before $e_i$ in the enumeration. 
\end{definition}

\begin{remark}\label{tatuvoador} To obtain an edge ordered enumeration of $\mathfrak p$ we proceed as follows. First, list $s(e_1), e_1, s(e_2), e_2$ (of course, $s(e_2)$ is only listed if it is different from $s(e_1)$). Next, without repetition, list all paths of length less or equal to $2$ generated by $\{e_1,e_2\}$; list, without repetition, $s(e_3), e_3$, followed by all paths of length less or equal to $3$ generated by $\{e_1,e_2,e_3\}$; continue the procedure successively.

\end{remark}

From now on we always assume that an edge-ordered enumeration of $\mathfrak p$ is picked.

\begin{example}\label{balaoespiao}
Let $E$ be the countable infinite graph with only one vertex and infinite edges $\{e_i\}$ (this graph, also called rose with infinite petals, is associated with $\mathcal{O}_\infty$, the Cuntz algebra generated by a countably infinite number of isometries). An edge-ordered enumeration of this graph is given below. \begin{align*}
\p=\{& s(e_1), e_1,  e_2, e_1e_1 , e_1e_2, e_2e_1, e_2e_2,  e_3, e_1e_3, e_2e_3, e_3e_1, e_3e_2, e_3e_3,\\
& e_1e_1e_1, e_1e_1e_2, e_1e_1e_3, e_1e_2e_1, e_1e_2e_2, e_1e_2e_3, e_1e_3e_1, e_1e_3e_2, e_1e_3e_3,\\ 
& e_2e_1e_1, e_2e_1e_2, e_2e_1e_3, e_2e_2e_1, e_2e_2e_2, e_2e_2e_3, e_2e_3e_1, e_2e_3e_2, e_2e_3e_3, \\ & e_3e_1e_1, e_3e_1e_2, e_3e_1e_3, e_3e_2e_1, e_3e_2e_2, e_3e_2e_3, e_3e_3e_1, e_3e_3e_2, e_3e_3e_3,\ldots\}. 
\end{align*}
\end{example}

\subsubsection*{Notation and properties of the metric.}

Next, we set up notation and describe some properties of the metric induced by an edge-ordered enumeration of $\mathfrak{p}$. We start with the definition of certain indexes, which we will use throughout the rest of the paper. 

\begin{definition}\label{pnk}
Let $(X,\sigma)$ be the Deaconu-Renault system associated with a graph $E$.  Endow $\mathfrak{p}$ with an edge-order enumeration as in Remark~\ref{tatuvoador}. For each $k \geq 1$, we define $N(k)$ as the index such that, in the enumeration of $\p$, $p_{N(k)}$ is the last path of length $k$ formed only by edges that belong to $\{e_{1}, e_{2}, e_{3}, \ldots, e_{k}\}$. 
\end{definition}

\begin{example}
    In Example~\ref{balaoespiao}, if  $k=2$ then $N(k)=7$.
\end{example}

\begin{definition} Let $(X,\sigma)$ be the Deaconu-Renault system associated with a graph $E$, $k \geq 1$, and $x$ and $y$ be distinct infinite paths in $X$. Define
\begin{eqnarray}
\mathfrak{i}_{xy} &:=& \mbox{min}\{i: x_{i} \neq y_{i}\},\\
\mathfrak{i}_{x}&:=& \mbox{min}\{i: x_{i} = e_{j}, j > k\}.
\end{eqnarray}
\end{definition}

\begin{remark}
    The index $\mathfrak i_x$ defined above depends on $k$ but, to keep the notation cleaner, we decided to denote without appending an extra $k$ to it. Moreover, as $x$ and $y$ are distinct infinite paths in $X$, the index $\mathfrak{i}_{xy}$ is well defined, but $\mathfrak{i}_{x}$ may not be defined (in the case when the set over which the minimum is taken is empty). 
\end{remark}

Using the above indexes, we describe below the inequality $d(x,y) < \dfrac{1}{2^{N(k)}}$, whenever $x$ and $y$ are distinct infinite paths of $X$.

\begin{proposition}\label{characteridelta}
Let $E$ be a graph and $(X,\sigma)$ be the Deaconu-Renault system associated with it. Let $k\geq 1$ and $x$, $y$ be distinct infinite paths of $X$. Then, $d(x,y) < \dfrac{1}{2^{N(k)}}$ if and only if $\mathfrak{i}_{xy} > k$, or both $\mathfrak i_x$ and $\mathfrak i_y$ are defined and $\mathfrak{i}_{x} = \mathfrak{i}_{y} \leq \mathfrak{i}_{xy}$. 
\end{proposition}
\begin{proof}
Let $x,y$, and $k$ be as in the hypothesis. Suppose that $\mathfrak{i}_{xy} > k$. Let $p_{i}$ be the first finite path in $\mathfrak{p}$ which is an initial segment of $x$ or $y$ but not the other. Then, $|p_{i}| > k$ and, by the definition of $N(k)$ and the properties of the edge-ordered enumeration of $\mathfrak{p}$, we have that $d(x,y) = \dfrac{1}{2^{i}} < \dfrac{1}{2^{N(k)}}$. Now, suppose that $\mathfrak{i}_{x} = \mathfrak{i}_{y} \leq \mathfrak{i}_{xy}$ and let $p_{j}$ be the first finite path in $\mathfrak{p}$ which is an initial segment of $x$ or $y$ but not the other. Then, $p_{j}$ is a finite path which contains an edge $e_{\ell}$, with $\ell > k$ and hence $d(x,y) = \dfrac{1}{2^{j}} < \dfrac{1}{2^{N(k)}}$. 

For the other implication, suppose that $x$ and $y$ are distinct infinite paths such that $d(x,y) < \dfrac{1}{2^{N(k)}}$. Suppose that $\mathfrak{i}_{xy} \leq k$. Clearly, if $\mathfrak i_x$ or $\mathfrak i_y$ is not defined then $d(x,y) \geq \dfrac{1}{2^{N(k)}}$, a contradiction. So, we have that both $\mathfrak i_x$ and $\mathfrak i_y$ are defined. We show that $\mathfrak{i}_{x} = \mathfrak{i}_{y} \leq \mathfrak{i}_{xy}$. Suppose that $\mathfrak{i}_{x} > \mathfrak{i}_{xy}$. Then, $x_{[1, \  \mathfrak{i}_{xy}]} = p_{j}$ is a finite path which is initial segment of $x$ but not of $y$. Moreover, since $|x_{[1, \  \mathfrak{i}_{xy}]}| = \mathfrak{i}_{xy}$ and  $x_{i} \in \{e_{j}: j \leq k\}$ for all $1\leq i \leq \mathfrak{i}_{xy}$, we have that $j \leq N(k)$ and $d(x,y) \geq \dfrac{1}{2^{j}}\geq \dfrac{1}{2^{N(k)}}$, a contradiction. Hence $\mathfrak{i}_{x} \leq \mathfrak{i}_{xy}$ and, analogously, $\mathfrak{i}_{y} \leq \mathfrak{i}_{xy}$. To finish, notice that if $\mathfrak{i}_{x} \neq \mathfrak{i}_{y}$ then $\mathfrak{i}_{xy}$ is not the minimal index where $x$ and $y$ differ, a contradiction. Therefore, $\mathfrak{i}_{x} = \mathfrak{i}_{y} \leq \mathfrak{i}_{xy}$ and the proof is finished. 
\end{proof}

\begin{remark}\label{carnaval} A priori, in the settings of Proposition~\ref{characteridelta}, we have that $d(x,y) \geq \dfrac{1}{2^{N(k)}}$ if, and only if, $\mathfrak{i}_{xy} \leq k$  and
\begin{enumerate}
    \item $\mathfrak{i}_{x}$ or $\mathfrak{i}_{y}$ is not defined, or
    \item  both $\mathfrak{i}_{x}$ and $\mathfrak{i}_{y}$ are defined, and $\mathfrak{i}_{x} > \mathfrak{i}_{xy}$ or $\mathfrak{i}_{y} > \mathfrak{i}_{xy}$, or
    \item both $\mathfrak{i}_{x}$ and $\mathfrak{i}_{y}$ are defined, $\mathfrak{i}_{x}, \mathfrak{i}_{y} \leq \mathfrak{i}_{xy}$, and $\mathfrak{i}_{x}\neq \mathfrak{i}_{y}$.
\end{enumerate}
Notice that Condition 3. above is always false (the existence of $\mathfrak{i}_{x}$ and $\mathfrak{i}_{y}$ as in Condition 3. contradicts the definition of $\mathfrak{i}_{xy}$). Hence, $d(x,y) \geq \dfrac{1}{2^{N(k)}}$ if, and only if, Conditions 1. and 2. above are satisfied.
\end{remark}

We finish this section observing that a number $\varepsilon>0$ determines a set of edges in an edge-ordered enumeration of $\p$. More precisely, we have the following definition.

\begin{definition}\label{Fepsilon}
 Let $(X,\sigma)$ be the Deaconu-Renault system associated with a graph $E$ and endow $\mathfrak{p}$ with an edge-ordered enumeration. Given $\varepsilon>0$, define \[F_\varepsilon = \{e\in E^1: e=p_j \in \mathfrak{p}, \text{ for some }  j \text{ such that } \frac{1}{2^j}\geq\varepsilon\}.\] 

\end{definition}

\begin{remark} Notice that $F_\varepsilon$ is always finite.     
\end{remark}

\section{Shadowing for Deaconu-Renault systems associated with graphs.}\label{capivara}

In this section, we characterize the finite shadowing property and the shadowing property, for the Deaconu-Renault system $(X,\sigma)$ associated with a graph, in terms of the existence of specific paths,  as we explain in  Definition~\ref{FPC1} (for the finite case) and in Definitions~\ref{FIPC} and \ref{SIPC}. 

We consider the space $X$ with the metric induced by an edge-ordered enumeration of $\mathfrak{p}$, as described in Definition~\ref{edgeordered} and Remark~\ref{tatuvoador}.

\subsection{Characterization of the finite shadowing property.}

The main goal of this section is to show that the Deaconu-Renault system associated with a graph satisfies the finite shadowing property if, and only if, it satisfies the Finite Path Condition (FPC) (see Definition~\ref{FPC1} and Theorem~\ref{teofiniteshadowing}). In a few words, given $\varepsilon>0$ and a finite set of finite paths $\{\lambda^{i}: 1 \leq i \leq \ell\}$ (which satisfies specific conditions, which depend on $\varepsilon$), the FPC allows us to find a finite path $\lambda$ that ``approximates'' the initial segments of the given paths.  

We start the section with a technical lemma, which says that when $x$ is an infinite path which (under some circumstances) does not $\frac{1}{2^{N(k)}}$-shadows a finite $\frac{1}{2^{N(k)}}$-chain $(x^{n})_{n=1}^{m}$, it is possible to find an index $s \in \mathbb{N}$ such that $x_{s} \neq x_{1}^{s}$.

\begin{lemma}\label{corNkfinite}
Fix $k \in \mathbb{N}$. Let $(x^{n})_{n=1}^{m}$ be a finite $\frac{1}{2^{N(k)}}$-chain of infinite paths. Suppose that there is an infinite sequence $x= x_{1}x_{2}x_{3}\ldots$ in $\p^\infty$, and a natural $n \geq 0$, such that $d(\sigma^{n}(x), x^{n+1}) \geq \dfrac{1}{2^{N(k)}}$. Let $a = \sigma^{n}(x)$, $b = x^{n+1}$ and $\mathfrak{i}_{ab} = \mbox{min}\{i: a_{i} \neq b_{i} \}$. If $n+\mathfrak{i}_{ab} \leq m$, then  $x_{n+\mathfrak{i}_{ab}} \neq x^{n+\mathfrak{i}_{ab}}_{1}$.  
\end{lemma}
\begin{proof}
Let $(x^{n})_{n=1}^m$, $x$, $b$ and $a$ be as in the hypothesis. Then, by Remark~\ref{carnaval}, we have that $\mathfrak{i}_{ab} \leq k$ and
 $\mathfrak{i}_{a}$ or $\mathfrak{i}_{b}$ is not defined, or $\mathfrak{i}_{ab} \leq k$ and both $\mathfrak{i}_{a}$ and $\mathfrak{i}_{b}$ are defined with $\mathfrak{i}_{a} > \mathfrak{i}_{ab}$ or $\mathfrak{i}_{b} > \mathfrak{i}_{ab}$.

Recall that, by the definition of $\mathfrak{i}_{ab}$, we have that $x_{n+\mathfrak{i}_{ab}} = \sigma^{n}(x)_{\mathfrak{i}_{ab}} \neq x^{n+1}_{\mathfrak{i}_{ab}}$ and, by hypothesis, we know that $n + \mathfrak{i}_{ab} \leq m$. Next, we split the proof into two parts.

Suppose first that $\mathfrak{i}_{b}$ is defined and $\mathfrak{i}_{b} \leq \mathfrak{i}_{ab}$. Then, either $\mathfrak{i}_{a} > \mathfrak{i}_{ab}$ or $i_{a}$ is not defined. In both cases, we have that $x_{n+\mathfrak{i}_{ab}} =\sigma^{n}(x)_{\mathfrak{i}_{ab}} \in \{e_{\ell}: \ell \leq k\}$. Also, by the minimality of $\mathfrak{i}_{ab}$, and because either $\mathfrak{i}_{b} < \mathfrak{i}_{a}$ or $i_{a}$ is not defined, we obtain that $\mathfrak{i}_{b} = \mathfrak{i}_{ab}$. As $(x^{n})_{n=1}^{m}$ is a finite $\frac{1}{2^{N(k)}}$-chain, $\mathfrak{i}_{ab} \leq k$ and $x^{n+1}_{\mathfrak{i}_{ab}} \in \{e_{j}: j > k\}$, we conclude that  $x^{n+2}_{\mathfrak{i}_{ab}-1} \in \{e_{j}: j > k\}$. Proceeding inductively, we obtain that $x^{n+i}_{\mathfrak{i}_{ab}-i+1} \in \{e_{j}: j > k\}$ for all $1 \leq i \leq \mathfrak{i}_{ab}$. In particular, $x_{1}^{n+\mathfrak{i}_{ab}} \in \{e_{j}: j > k\}$ and hence $x_{n+\mathfrak{i}_{ab}} =\sigma^{n}(x)_{\mathfrak{i}_{ab}} \neq x_{1}^{n+\mathfrak{i}_{ab}}$, as desired.

Next, suppose that $\mathfrak{i}_{b}$ is defined and $\mathfrak{i}_{b} > i_{ab}$ or that $\mathfrak{i}_{b}$ is not defined. Then, $x^{n+1}_{i} \in \{e_{\ell}: \ell \leq k\}$ for all $1 \leq i \leq \mathfrak{i}_{ab}$ and, as $\mathfrak{i}_{ab} \leq k$ and $(x^{n})_{n=1}^{m}$ is a finite $\frac{1}{2^{N(k)}}$-chain, we have that $x^{n+2}_{\mathfrak{i}_{ab}-1} = x^{n+1}_{\mathfrak{i}_{ab}}$. Proceeding inductively, we conclude that the equality $x^{n+i}_{\mathfrak{i}_{ab} - i +1 } = x^{n + 1}_{\mathfrak{i}_{ab}}$ holds for all $1 \leq i \leq \mathfrak{i}_{ab}$. In particular, $x^{n+1}_{\mathfrak{i}_{ab}} = x^{n+\mathfrak{i}_{ab}}_{1}$. Then, since $a_{\mathfrak{i}_{ab}} \neq b_{\mathfrak{i}_{ab}}$, we conclude that $x_{n+\mathfrak{i}_{ab}} \neq x^{n+1}_{\mathfrak{i}_{ab}} = x^{n+\mathfrak{i}_{ab}}_{1}$, and the proof is finished.
\end{proof}

With the above lemma, we can show that it is always possible to shadow a finite $\frac{1}{2^{N(k)}}$-chain $(x^{n})_{n=1}^{m}$ which satisfies $\sigma(x^{n})_{1} = x^{n+1}_{1}$ for all $1 \leq n < m$. This is the content of the proposition below.

\begin{proposition}\label{nkfinitechain}
Fix $k \in \mathbb{N}$, $k \geq 1$. Let $(x^{n})_{n=1}^{m}$ be a finite $\frac{1}{2^{N(k)}}$-chain, formed by infinite paths, which satisfies $x^{n}_{2} = x^{n+1}_{1}$ for all natural $n$ with $1 \leq n < m$. Then, there exists an infinite path $x \in \mathfrak p^\infty$ which satisfies the following two conditions:
\begin{enumerate}
    \item $x_{i} = x^{i}_{1}$, for all $1 \leq i \leq m$;
    \item If $x^{m}_{1} \in F_{\frac{1}{2^{N(k)}}}$, then $x_{m+j} = x^{m}_{j+1}$ for all $0 \leq j < k$.
    \end{enumerate}
Moreover, any infinite path that satisfies the conditions 1 and 2 above $\frac{1}{2^{N(k)}}$-shadows $(x^{n})_{n=1}^{m}$.
\end{proposition}
\begin{proof}
Let $(x^{n})_{n=1}^{m}$ be a finite $\frac{1}{2^{N(k)}}$-chain according to the hypothesis. Since $x^{n}_{2} = x^{n+1}_{1}$ for all $1 \leq n < m$, we have that $s(x^{n+1}_{1}) = r(x^{n}_{1})$ for all $1 \leq n < m$. This implies the existence of an infinite path $x \in \p^\infty$ which satisfies Conditions 1 and 2 above.

Next, we show that any $x \in \p^\infty$ which satisfies Conditions 1. and 2. $\frac{1}{2^{N(k)}}$-shadows $(x^{n})_{n=1}^{m}$. Suppose that this is not true, that is, assume that there is an infinite path $x \in \p^\infty$ which satisfies conditions 1 and 2, and a natural $0 \leq n < m$ such that $d(\sigma^{n}(x),x^{n+1}) \geq \dfrac{1}{2^{N(k)}}$. Let $a= \sigma^{n}(x)$ and $b = x^{n+1}$.

Suppose that $n+\mathfrak{i}_{ab}\leq m$. Applying Lemma~\ref{corNkfinite} for $a$ and $b$ above, we obtain that $x_{n+\mathfrak{i}_{ab}} \neq x^{n+\mathfrak{i}_{ab}}_{1}$, which contradicts Condition 1.  

Next, suppose that $n+\mathfrak{i}_{ab}> m$. Let $i_{0}:= m - n$ and notice that $i_0 < \mathfrak{i}_{ab}$. By the definition of $\mathfrak{i}_{ab}$, we have that $a_{i_{0}} = b_{i_{0}}$. Moreover, as $d(a,b) \geq \dfrac{1}{2^{N(k)}}$, we obtain that $a_{i}, b_{i} \in F_{\frac{1}{2^{N(k)}}}$ and $a_{i} = b_{i}$ for all $1\leq i < \mathfrak{i_{ab}}$. Since $(x^{n})_{n=1}^{m}$ is a finite $\frac{1}{2^{N(k)}}$-chain, we must have $x^{n+2}_{[1,\ \mathfrak{i}_{ab}-2]} = x^{n+1}_{[2, \ \mathfrak{i}_{ab}-1 ]}$, $x^{n+3}_{[1,\ \mathfrak{i}_{ab}-3]} = x^{n+1}_{[3, \ \mathfrak{i}_{ab}-1 ]}$ and, inductively, $x^{n+j}_{[1,\ \mathfrak{i}_{ab}-j]} = x^{n+1}_{[j, \ \mathfrak{i}_{ab}-1 ]}$ for all $1 \leq j \leq m - n$. In particular, $x^{m}_{1} \in F_{\frac{1}{2^{N(k)}}}$. Now, we split the proof into two cases. First, suppose that $x^{n+1}_{\mathfrak{i}_{ab}} \in F_{\frac{1}{2^{N(k)}}}$. Again, since $(x^{n})_{n=1}^{m}$ is a finite $\frac{1}{2^{N(k)}}$-chain, we obtain that $x^{n+j}_{\mathfrak{i}_{ab}-j+1} = x^{n+1}_{\mathfrak{i}_{ab}}$ for all $1 \leq j \leq m - n$. In particular, $x^{m}_{\mathfrak{i}_{ab}-m+n+1} = x^{n+1}_{\mathfrak{i}_{ab}}$. This means that $x_{m+j_{0}} \neq  x^{m}_{j_{0}+1}$, when $j_{0} = \mathfrak{i}_{ab} - m + n$. As $j_{0} < k$ (because $n < m$ and $\mathfrak{i_{ab}} \leq k$) we have a contradiction with Condition 2. We are left with the case when $x^{n+1}_{\mathfrak{i}_{ab}} \notin F_{\frac{1}{2^{N(k)}}}$. As $d(\sigma^{n}(x), x^{n+1}) \geq \dfrac{1}{2^{N(k)}}$, we must have that $x_{n+\mathfrak{i}_{ab}} \in F_{\frac{1}{2^{N(k)}}}$. But, since $(x^{n})_{n=1}^{m}$ is a finite $\frac{1}{2^{N(k)}}$-chain, we must also have that $x^{n+j}_{\mathfrak{i}_{ab}-j+1} \notin  F_{\frac{1}{2^{N(k)}}}$  for all $1 \leq j \leq m-n$. In particular, $x_{m+j_{0}} \neq  x^{m}_{j_{0}+1}$, for $j_{0} = \mathfrak{i}_{ab} - m + n$. As we already know that $j_{0} < k$, we have again a contradiction Condition 2. 

We proved above that the assumption $d(\sigma^{n}(x),x^{n+1}) \geq \dfrac{1}{2^{N(k)}}$ for some $1 \leq n < m$ leads to a contradiction. Therefore, we have that $d(\sigma^{n}(x),x^{n+1}) < \dfrac{1}{2^{N(k)}}$ for all $1 \leq n < m$, as desired. 
\end{proof}

Now we move to the final preparations to Theorem \ref{teofiniteshadowing}, where we deal with the finite shadowing property. Below we define a function $f$ which will be used throughout the remainder of the section. In particular, this function is used in Definition~\ref{FPC1}, where we introduce the Finite Path Condition (which we show, in Theorem~\ref{teofiniteshadowing}, to be equivalent to the finite shadowing property). Roughly speaking, if we enumerate the entries of the sequences (except the last term of each sequence) of a $\delta-$chain in a
natural way, then the function $f$ gives the position of the entries of the sequences that form the $\delta-$chain. 

\begin{definition}\label{f(in)}
Let $\lambda^{1}, \lambda^{2}, \ldots, \lambda^{\ell}$ be finite paths such that $|\lambda^{\ell}| \geq 1$ and $|\lambda^{i}| \geq 2$ if $i < \ell$. Set $|\lambda^{0}| := 0$ and define  $f:\bigcup_{i=1}^{\ell}\{(i,n)\in \N\times \N: \ 1 \leq n < |\lambda^{i}|\} \cup \{(\ell, |\lambda^{\ell}|)\} \rightarrow \mathbb{N}$ by
\[    f(i,n)= \displaystyle\left(\sum_{j=0}^{i-1}|\lambda^{j}| \right)+ n -i +1.\]
 
\end{definition}

\begin{example}
    Let $\lambda^1 = \lambda^1_1 \lambda^1_2$, $\lambda^2 = \lambda^2_1 \lambda^2_2 \lambda^2_3$ and $\lambda^3 = \lambda^3_1 \lambda^3_2 \lambda^3_3 \lambda^3_4$. Then, $f(1,1)=1, f(2,1)=2, f(2,2)=3, f(3,1)=4, f(3,2)=5, f(3,3)=6,$ and $f(3,4)=7.$
\end{example}

Below we observe that $f$ is injective and describe its image.

\begin{lemma}\label{remoafogado}
Let $\lambda^{1}, \lambda^{2}, \ldots, \lambda^{\ell}$ be finite paths as in Definition~\ref{f(in)} above. Then, $f$ is injective and $Im(f)=\{m\in \N:1 \leq m \leq \displaystyle \left(\sum_{j=0}^{\ell}|\lambda^{j}|\right) - \ell + 1 \}$.
\end{lemma}

\begin{proof}
First we prove that $f(i,n) < f(i+1,1)$ for all $1 \leq i < \ell$ and $1 \leq n < |\lambda^{i}|$  (notice that this implies that $f(i,n) < f(j,1)$ for all $1 \leq i<j\leq \ell$ and $1 \leq n < |\lambda^{i}|$).  Indeed,
\begin{align*}
    f(i,n) & =   n -i +1+\displaystyle\sum_{j=0}^{i-1}| \lambda^{j}| <   \displaystyle |\lambda^{i}| -i +1 +\sum_{j=0}^{i-1}|\lambda^{j}|  = \\
           & = \displaystyle  1 - (i+1) + 1 + \sum_{j=0}^{i}|\lambda^{j}| 
           = f(i+1,1).
\end{align*}
Now, let $(i_{0},n_{0})$ and $(i_{1},n_{1})$ be pairs such that
\begin{equation}\label{eqf}
    f(i_{0},n_{0})=f(i_{1},n_{1})
\end{equation}
If $i_{0} \neq i_{1}$, say $i_{0}<i_{1}$, then $f(i_{0},n_{0})  <  f(i_{1},1) \leq f(i_{1},n_{1})$, which contradicts equation (\ref{eqf}). Hence $i_{0} = i_{1}$ and, by the definition of $f$ and equation (\ref{eqf}), we obtain that $n_{0} = n_{1}$. Thus, $f$ is injective and we have the following equalities: \[\#Im(f) = \#Dom(f) = \displaystyle \left( \sum_{j=0}^{\ell}|\lambda^{j}| \right)- \ell + 1.\] Since $f(1,1) = 1$ and
\[f(\ell, |\lambda^{\ell}|) = \displaystyle \left(\sum_{j=0}^{\ell}|\lambda^{j}|\right)  - \ell + 1 = \mbox{max}\{f(i,n): (i,n) \in Dom(f)\}, \]
we conclude that for each $1 \leq m \leq \displaystyle\left( \sum_{j=0}^{\ell}|\lambda^{j}|\right) - \ell + 1$, there is a unique pair $(i,n) \in Dom(f)$ such that $f(i,n)=m$, as we wanted. 

\end{proof}

Next, we describe a condition that is equivalent to the finite shadowing property (see Theorem \ref{teofiniteshadowing}). The idea of this condition is that given an $\varepsilon>0$, there is a $\delta>0$ such that for any set of finite paths $\lambda^{1}, \lambda^{2}, \ldots, \lambda^{\ell}$, we can find a finite path $\lambda$ such that, for each $1 \leq j \leq |\lambda| = f(\ell, |\lambda^{\ell}|)$, we have $\lambda_{j} = \lambda^{i}_{n}$ if $\lambda^{i}_{n} \in F_{\varepsilon}$ and $\lambda_{j} \notin F_{\varepsilon}$ if $\lambda^{i}_{n} \notin F_{\varepsilon}$, where $(i,n)$ is the unique pair such that $f(i,n)=j$. Roughly, $\lambda_{j}$ needs to be close to $\lambda^{i}_{n}$ when $j = f(i,n)$. 

\begin{definition}\label{FPC1}\textbf{(Finite Path Condition - FPC)} The Deaconu-Renault system $(X,\sigma)$ associated with a graph satisfies the Finite Path Condition if, for all $\varepsilon > 0$, there is $0 < \delta \leq \varepsilon$ such that for all finite set $\lambda^{1}, \lambda^{2}, \ldots, \lambda^{\ell}$ of finite paths which satisfy:
\begin{enumerate}
    \item $|\lambda^{i}| \geq 2$ for $ 1 \leq i \leq \ell$;
    \item $(F_{\delta})^{c} \cap \left\{\{\lambda^{i}_{j}: \  1 \leq i\leq \ell, \ 1 \leq j \leq |\lambda^{i}| \}\setminus \{\lambda^{1}_{1}\}\right\}= \{\lambda^{i+1}_{1}: 1 \leq i < \ell\} \cup \{\lambda^{i}_{|\lambda^{i}|}: 1 \leq i < \ell\}$,
\end{enumerate}
    there is a finite path $\lambda$ such that $\displaystyle |\lambda| = \left(\sum_{i=1}^{\ell}|\lambda^{i}|\right)-\ell + 1$
     and such that, for all $i,n$ for which $(i,n)\in Dom(f)$ (see Definition~\ref{f(in)}), the following conditions are satisfied:
     
\begin{enumerate}[a.]
        \item If $\lambda_{n}^{i} \in F_{\varepsilon}$, then $\lambda_{f(i,n)} = \lambda^{i}_{n}$;
        \item If $\lambda_{n}^{i} \notin F_{\varepsilon}$, then $\lambda_{f(i,n)} \notin F_{\varepsilon}$.
\end{enumerate}

\end{definition}

Before we prove that the Finite Path Condition defined above is equivalent to the finite shadowing property, we need one last auxiliary result.


\begin{lemma}\label{lemmafinfinite}
Let $\varepsilon>0$ and $0<\delta\leq \varepsilon$. Suppose that $\lambda^{1}, \lambda^{2}, \ldots, \lambda^{\ell}$ are finite paths which satisfy Conditions~1. and 2. of Definition~\ref{FPC1}. For each $1 \leq i \leq \ell$, let $y^{i}$ be an infinite path such that $y^{i}_{k} = \lambda^{i}_{k}$ for all $1 \leq k \leq |\lambda^{i}|$ and, 
for each $i,n$ such that $(i,n)\in Dom(f)$, let $$x^{f(i,n)} := \sigma^{n-1}(y^{i}).$$
Then, $(x^{f(i,n)})_{f(i,n)=1}^{f(\ell,|\lambda^{\ell}|)}$ is a finite $\delta$-chain and every infinite path $x \in X$  which $\varepsilon$-shadows such chain satisfies:
\begin{enumerate}[a.]
    \item If $\lambda_{n}^{i} \in F_{\varepsilon}$, then $x_{f(i,n)} = \lambda^{i}_{n}$, for each $i,n$ such that $(i,n)\in Dom(f)$;
        \item If $\lambda_{n}^{i} \notin F_{\varepsilon}$, then $x_{f(i,n)} \notin F_{\varepsilon}$, for each $i,n$ such that $(i,n)\in Dom(f)$. 
\end{enumerate}

\end{lemma}
\begin{proof}

 We begin proving that, for each $f(i,n)$, with $1 \leq f(i,n) \leq \displaystyle f(\ell,|\lambda^{\ell}|)$, we have that $d\left(\sigma(x^{f(i,n)}), \ x^{f(i,n)+1}\right)<\delta$. 
 
 First, suppose that $1 \leq i < \ell$ and $n=|\lambda^{i}|-1$. Then, we have that
 \begin{align*}
     f(i,|\lambda^{i}|-1)+1 & = \displaystyle \left( \sum_{j=0}^{i-1}|\lambda^{j}|\right) + |\lambda^{i}|-1 - i + 1 + 1\\
     & = \displaystyle\left(\sum_{j=0}^{i}|\lambda^{j}|\right) - i -1 + 1 + 1\\ 
     & = f(i+1,1).
 \end{align*}
 Therefore, $\sigma(x^{f(i, \ |\lambda^{i}|-1)})_{1} = \sigma(\sigma^{|\lambda^{i}|-1}(y^{i}))_{1}=\sigma^{|\lambda^{i}|-1}(y^{i})_{1} = y^{i}_{|\lambda^{i}|} \in (F_{\delta})^{c}$ and, moreover, $x^{f(i,\ |\lambda^{i}|-1)+1}_{1} = x^{f(i+1,\ 1)}_{1} = y^{i+1}_{1} \in (F_{\delta})^{c}$. From this, we conclude that $d\left(\sigma(x^{f(i,n)}), \ x^{f(i,n)+1}\right)<\delta$. 
 
Next, suppose that $1 \leq i \leq \ell$ and $1 \leq n < |\lambda^{i}|-1$. Then, 
 \begin{align*}
     f(i,n)+1 & = \displaystyle\left(\sum_{j=0}^{i-1}|\lambda^{j}|\right) + n - i + 1 + 1\\
     & = \displaystyle\left(\sum_{j=0}^{i-1}|\lambda^{j}|\right)+ n +1 -i + 1\\ 
     & = f(i,n+1).
 \end{align*}
 Hence, $\sigma(x^{f(i,n)})=\sigma(\sigma^{n-1}(y^{i})) = \sigma^{n}(y^{i}) = x^{f(i,\ n+1)}$, and so $d\left(\sigma(x^{f(i,n)}),\ x^{f(i,\ n+1)}\right)< \delta$.
 
 For the third and final case, suppose that $i=\ell$ and $n = |\lambda^{\ell}|-1$. As
 \begin{align*}
     f(\ell,|\lambda^{\ell}|-1)+1 & = \displaystyle\left(\sum_{j=0}^{\ell-1}|\lambda^{j}| \right)+ |\lambda^{\ell}|-1 - \ell + 1 + 1\\
     & = \displaystyle\left(\sum_{j=0}^{\ell}|\lambda^{j}|\right)  -\ell + 1\\ 
     & = f(\ell,|\lambda^{\ell}|),
 \end{align*}
 we have that $\sigma(x^{f(\ell, \ |\lambda^{\ell}|-1)}) = \sigma(\sigma^{|\lambda^{\ell}|-2}(y^{\ell}))=\sigma^{|\lambda^{\ell}|-1}(y^{\ell})=x^{f(\ell,\ |\lambda^{\ell}|)}=x^{f(\ell,\ |\lambda^{\ell}|-1)+1}$. Hence, $d\left(\sigma(x^{f(\ell,\ |\lambda^{\ell}|-1)}), \ x^{f(\ell,\ |\lambda^{\ell}|-1)+1}\right)<\delta$.
 
 Summarizing the above, we have proved that $(x^{f(i,n)})_{f(i,n)=1}^{f(\ell,|\lambda^{\ell}|)}$, is a $\delta$-chain. 
 
 For the last part, suppose that $x \in X$ is an infinite path which $\varepsilon$-shadows $(x^{f(i,n)})_{f(i,n)=1}^{f(\ell,|\lambda^{\ell}|)}$. Then, $d\left(\sigma^{f(i,n)-1}(x),x^{f(i,n)}\right) < \varepsilon$, for all $1 \leq f(i,n) \leq f(\ell, |\lambda^{\ell}|)$. Notice that $x^{f(i,n)}_{1} = \sigma^{n-1}(y^{i})_{1} = y^{i}_{n} = \lambda^{i}_{n}$. Hence, if $\lambda^{i}_{n} \in F_{\varepsilon}$ we must have that $\sigma^{f(i,n)-1}(x)_{1} = \lambda^{i}_{n}$ (because $d\left(\sigma^{f(i,n)-1}(x),x^{f(i,n)}\right) < \varepsilon$ and $x^{f(i,n)}_{1} = \lambda^{i}_{n}$). But $\sigma^{f(i,n)-1}(x)_{1} = x_{f(i,n)}$ and so item a. follows. On the other hand, if $\lambda^{i}_{n} \notin F_{\varepsilon}$ then we must have that $\sigma^{f(i,n)-1}(x)_{1} \notin F_{\varepsilon}$. Hence item b., and the lemma, are proved.
\end{proof}

We now show that the finite shadowing property is equivalent to the finite path condition.

\begin{theorem}\label{teofiniteshadowing}
Let $(X,\sigma)$ be the Deaconu-Renault system associated with a graph. Then, $X$ has the finite shadowing property if, and only if, it satisfies the Finite Path Condition (Definition \ref{FPC1}).
\end{theorem}
\begin{proof}
Suppose that $(X,\sigma)$ has the finite shadowing property. Given $\varepsilon > 0$, let $0<\delta \leq \varepsilon$ be such that every finite $\delta$-chain is $\varepsilon$-shadowed by an infinite path $x \in X$ (notice that since we assume that our graphs have no sinks, every finite path can be extended to an infinite path). Let $\lambda^{1}, \lambda^{2}, \ldots, \lambda^{\ell}$ be finite paths which satisfy Conditions~1. and 2. of Definition \ref{FPC1} and construct a $\delta$-chain, $(x^{f(i,n)})_{f(i,n)=1}^{f(\ell,|\lambda^{\ell}|)}$, as described in Lemma~\ref{lemmafinfinite}. 
Let $x \in X$ be an infinite path that $\varepsilon$-shadows it. Define the finite path $\lambda$, of length $|\lambda| = f(\ell,|\lambda^{\ell}|)$, by $\lambda_{j}:=x_{j}$, for all $1 \leq j \leq f(\ell,|\lambda^{\ell}|)$. By Lemma~\ref{lemmafinfinite}, such finite path satisfies items a. and b. of Definition~\ref{FPC1}. Therefore, $X$ satisfies the Finite Path Condition and we have finished the first part of the proof.

For the converse, suppose that $(X,\sigma)$  satisfies the Finite Path Condition. Fix $\varepsilon > 0$. Let $k$ be a natural such that $\frac{1}{2^{N(k)}} \leq \varepsilon$ (where $N(k)$ is given in Definition \ref{pnk}) and let $\varepsilon'$ be a real number such that $\varepsilon'<\frac{1}{2^{N(k)}}$. Let $\delta > 0$ be the real number of the Finite Path Condition (related to $\varepsilon'$) and let $(x^{n})_{n=1}^{m}$ be a finite $\delta$-chain.

Suppose that $x_{1}^{n} \in F_{\delta}$ for all $n \geq 2$. In this case, we must have that $\sigma(x^{n})_{1} = x^{n+1}_{1}$ for all $1 \leq n < m$, since $(x^{n})_{n=1}^{m}$ is a finite $\delta$-chain. As $(x^{n})_{n=1}^{m}$ is also a $\frac{1}{2^{N(k)}}$-chain, by Proposition~\ref{nkfinitechain}, we obtain an infinite path $x \in X$ that $\frac{1}{2^{N(k)}}$-shadows such chain and hence also $\varepsilon$-shadows it (since $\frac{1}{2^{N(k)}} \leq \varepsilon$). 

Next, assume that there is an $n \geq 2$ such that $x^{n}_{1} \notin F_{\delta}$. Consider the set 
\begin{equation}
A_{\delta} = \{n \geq 2: x_{1}^{n} \notin F_{\delta}\} = \{n_{1}, n_{2}, \ldots, n_{\ell-1}\}.    
\end{equation}
Notice that $\sigma(x^{n_{i}-1})_{1} \notin F_{\delta}$ for $1 \leq i \leq \ell-1$ (because $d\left(\sigma(x^{n_{i}-1}), x^{n_{i}})\right)< \delta$ and $x^{n_{i}}_{1} \notin F_{\delta}$). Also, if $n \geq 2$ and $n \notin A_{\delta}$, then $\sigma(x^{n-1})_{1}= x^{n}_{1}$ (because $x^{n}_{1} \in F_{\delta}$ and $d\left(\sigma(x^{n-1}),x^{n}\right) < \delta$). Given this information, making $n_{0}:=1$, define the finite paths $\lambda^{1},\lambda^{2},\ldots,\lambda^{\ell}$ by:
\begin{equation}\label{lambdaix}
\lambda^{i}=
		\left\{
		\begin{array}{ll}
		x^{n_{i-1}}_{1}x^{n_{i-1}+1}_{1}\ldots x^{n_{i}-1}_{1}\sigma(x^{n_{i}-1})_{1}, & \text{if $1\leq i<\ell$};\\
		x^{n_{\ell-1}}_{1}x^{n_{\ell-1}+1}_{1}\ldots x^{m}_{1}, & \text{if $i=\ell$}.\\
				\end{array}
		\right.    
\end{equation}

 Notice that:
\begin{equation}\label{|lambdaix|}
|\lambda^{i}|=
		\left\{
		\begin{array}{ll}
		n_{i}-n_{i-1}+1, & \text{if $1\leq i<\ell$};\\
		m- n_{\ell-1}+1, & \text{if $i=\ell$}.\\
				\end{array}
		\right.    
\end{equation}
By (\ref{lambdaix}) and (\ref{|lambdaix|}) above, and Definition \ref{f(in)}, we have that (for $i,n$ such that $(i,n)\in Dom (f)$):
\begin{equation}\label{f(inni)}
f(i,n)=\displaystyle\left(\sum_{j=0}^{i-1}|\lambda^{j}|\right) + n -i +1= n_{i-1}+n-1.
		\end{equation}
Now, by (\ref{lambdaix}) and (\ref{f(inni)}) above, we have that (also for $i,n$ such that $(i,n)\in Dom (f)$): 
\begin{equation}\label{lambdaxf}
    \lambda^{i}_{n}=x^{n_{i-1}+n-1}_{1}=x^{f(i,n)}_{1}.
\end{equation}

Notice that the finite paths $\lambda^{1},\lambda^{2},\ldots, \lambda^{\ell}$ satisfy the first and second conditions of Definition~\ref{FPC1}. 
Then, as $(X,\sigma)$ satisfies the Finite Path Condition (Definition \ref{FPC1}), there exists a finite path $\lambda$ such that
$|\lambda| = \displaystyle\left(\sum_{i=1}^{\ell}|\lambda^{i}|\right) -\ell + 1 = m$ and conditions a. and b. of the Finite Path Condition (Definition~\ref{FPC1}) are satisfied.

Let $x \in X$ be an infinite path such that 
\begin{equation}\label{equationx}
    x_{[1,\ m]} = \lambda \ \text{and, if} \ x_{1}^{m} \in F_{\varepsilon'} \ \text{then} \ x_{m+j}=\sigma^{j}(x^{m})_{1} = x^{m}_{j+1} \ \forall \ 0 \leq j \leq k.
\end{equation} 

Notice that, by (\ref{lambdaxf}), (\ref{equationx})  and item a. of the Finite Path Condition we have that:
\begin{equation}\label{xlambdaxi}
    x_{f(i,n)}=\lambda_{f(i,n)}=\lambda^{i}_{n}=x_{1}^{f(i,n)}, \  \text{if $\lambda_{n}^{i} \in F_{\varepsilon'}$, }
\end{equation}
Moreover, by Equation (\ref{xlambdaxi}), we have that $x_{s} = x^{s}_{1}$ for all $1 \leq s \leq f(\ell,|\lambda^{\ell}|)$ such that $s = f(i,n)$ and $\lambda_{n}^{i} \in F_{\varepsilon'}$.
Also, by (\ref{lambdaxf}), (\ref{equationx})  and item b. of the Finite Path Condition we have that $x_{s} \notin F_{\varepsilon'}$ if $x^{s}_{1} \notin F_{\varepsilon'}$, for all $1 \leq s \leq f(\ell,|\lambda^{\ell}|)$.

Finally, we prove that $x$ $\varepsilon'$-shadows $(x^{n})_{n=1}^{m}$ (and hence also $\varepsilon$-shadows, since $\varepsilon' \leq \varepsilon$). Fix a positive integer $0 \leq s < m$ and consider $a:=\sigma^{s}(x)$ and $b:=x^{s+1}$. Next, we split the proof into a few cases and check that $d(a,b)< \varepsilon'$ in all of them.

If $a=b$ then clearly $d(a,b) < \varepsilon'$. Suppose that $a \neq b$. Recall that  $\mathfrak{i}_{ab}= \mbox{min}\{i: a_{i} \neq b_{i}\}$. If $\mathfrak{i}_{ab} > k$, then $d(a,b) < \varepsilon'$, see Definition~\ref{pnk} and Proposition~\ref{characteridelta} for further details. So, suppose that $\mathfrak{i}_{ab} \leq k$. If both $a_{\mathfrak{i}_{ab}}$ and $b_{\mathfrak{i}_{ab}}$ do not belong to $F_{\varepsilon'}$, then we have that $d(a,b)< \varepsilon'$ (because any finite path which is an initial segment of $a$ or $b$, but not the other, must contain an edge which is not in $F_{\varepsilon'}$). We are left with the case where  $b_{\mathfrak{i}_{ab}} \in F_{\varepsilon'}$ and with the case where $a_{\mathfrak{i}_{ab}} \in F_{\varepsilon'}$.  In both situations, to conclude that $d(a,b) < \varepsilon'$, we must prove that there is $1 \leq i < \mathfrak{i}_{ab}$ such that $a_{i}=b_{i} \notin F_{\varepsilon'}$.

Suppose that $x^{s+1}_{\mathfrak{i}_{ab}}=b_{\mathfrak{i}_{ab}} \in F_{\varepsilon'}$. We also assume that $s+\mathfrak{i}_{ab}\leq m$. Suppose that for all $1 \leq i < \mathfrak{i}_{ab}$ we have $a_{i}=b_{i} \in F_{\varepsilon'}$. As $(x^{n})_{n=1}^{m}$ is a $\varepsilon'$-chain, we must have $x^{s+j}_{1} = x^{s+1}_{j}$ for all $1 \leq j \leq \mathfrak{i}_{ab}$. In particular, $x^{s+\mathfrak{i}_{ab}}_{1} = x^{s+1}_{\mathfrak{i}_{ab}} = b_{\mathfrak{i}_{ab}} \in F_{\varepsilon'}$. Then, $a_{\mathfrak{i}_{ab}} = x_{s+\mathfrak{i}_{ab}} \neq x^{s+\mathfrak{i}_{ab}}_{1} \in F_{\varepsilon'}$, which contradicts Equation (\ref{xlambdaxi}). Hence, there exists an index $i$, $1 \leq i < \mathfrak{i}_{ab}$, such that $a_{i}=b_{i} \notin F_{\varepsilon'}$ and so $d(a,b)< \varepsilon'$. Now, we assume that $s+\mathfrak{i}_{ab} > m$. Suppose that for all $1 \leq i < \mathfrak{i}_{ab}$ we have $a_{i}=b_{i} \in F_{\varepsilon'}$. As $(x^{n})_{n=1}^{m}$ is a $\varepsilon'$-chain, and $\mathfrak{i}_{ab} \leq k,$ we must have $x^{s+j}_{[1, \ \mathfrak{i}_{ab} - j +1]} = x^{s+1}_{[j, \ \mathfrak{i}_{ab}]}$ for all $1 \leq j \leq m-s$. In particular, $x^{m}_{\mathfrak{i}_{ab} - m + s +1} = x^{s+1}_{\mathfrak{i}_{ab}}$. Hence $x^{m}_{\mathfrak{i}_{ab} - m + s +1} = x^{s+1}_{\mathfrak{i}_{ab}}\neq x_{n+\mathfrak{i}_{ab}} = x_{m+(\mathfrak{i}_{ab}-m+s)}$ and, as $j:= \mathfrak{i}_{ab}-m+s < k$, this contradicts (\ref{equationx}). Therefore, once again we conclude that there exists an index $i$, $1 \leq i < \mathfrak{i}_{ab}$, such that $a_{i}=b_{i} \notin F_{\varepsilon'}$, and so $d(a,b)< \varepsilon'$.

Finally, suppose that $\sigma^{s}(x)_{\mathfrak{i}_{ab}}=a_{\mathfrak{i}_{ab}}\in F_{\varepsilon'}$. We also suppose that $x^{s+1}_{\mathfrak{i}_{ab}}=b_{\mathfrak{i}_{ab}} \notin F_{\varepsilon'}$ (since we have dealt with the case $b_{\mathfrak{i}_{ab}} \in F_{\varepsilon'}$ in the previous paragraph). Assume that $s+\mathfrak{i}_{ab}\leq m$ and suppose that for all $1 \leq i < \mathfrak{i}_{ab}$ we have $a_{i}=b_{i} \in F_{\varepsilon'}$. As $(x^{n})_{n=1}^{m}$ is a $\varepsilon'$-chain, we must have $x^{s+j}_{\mathfrak{i}_{ab}-j+1} \notin F_{\varepsilon'}$ for all $1 \leq j \leq \mathfrak{i}_{ab}$. In particular, $x^{s+\mathfrak{i}_{ab}}_{1} \notin F_{\varepsilon'}$. Then, $a_{\mathfrak{i}_{ab}} = x_{s + \mathfrak{i}_{ab}} \neq x^{s+\mathfrak{i}_{ab}}_{1}$, contradicting Equation (\ref{xlambdaxi}). Therefore, there exists an index $i$, $1 \leq i < \mathfrak{i}_{ab}$, such that $a_{i}=b_{i} \notin F_{\varepsilon'}$ and hence $d(a,b)< \varepsilon'$. Now, assume that $s + \mathfrak{i}_{ab} > m$. Suppose that for all $1 \leq i < \mathfrak{i}_{ab}$ we have $a_{i}=b_{i} \in F_{\varepsilon'}$. As $(x^{n})_{n=1}^{m}$ is a $\varepsilon'$-chain, and $\mathfrak{i}_{ab} \leq k,$ we must have that $x^{s+j}_{\mathfrak{i}_{ab}-j+1} \notin F_{\varepsilon'}$ for all $1 \leq j \leq m-s$. In particular, $x^{m}_{\mathfrak{i}_{ab}-m+s+1} \notin F_{\varepsilon'}$. Then, we have that $x_{m+(\mathfrak{i}_{ab}-m+s)} \neq x^{m}_{(\mathfrak{i}_{ab}-m+s)+1}$ (because $x_{m+(\mathfrak{i}_{ab}-m+s)}=x_{s + \mathfrak{i}_{ab}} \in F_{\varepsilon'}$ and $x^{m}_{(\mathfrak{i}_{ab}-m+s)+1} \notin F_{\varepsilon'}$). As $j:= \mathfrak{i}_{ab}-m+s < k$, this contradicts (\ref{equationx}). Therefore, once more we obtain that there exists an index $i$, $1 \leq i < \mathfrak{i}_{ab}$, such that $a_{i}=b_{i} \notin F_{\varepsilon'}$, which implies that $d(a,b)< \varepsilon'$. We conclude that $x$ $\varepsilon'$-shadows the finite $\varepsilon'$-chain $(x^{n})_{n=1}^{m}$, as desired.
\end{proof}

\subsection{Characterization of the infinite shadowing property.}

In this section, we characterize the infinite shadowing property in terms of two conditions, which we call the first and second infinite path conditions. 

Throughout this section,  $(x^{n})_{n}$ denotes an infinite chain. We start with a proposition that, although simple, is essential in the sequel (in particular to show that, for $\delta = \dfrac{1}{2^{N(k)}}$ with $k\in \N$ (see Definition \ref{pnk}), it is always possible to find an element that $\delta$-shadows a $\delta$-chain $(x^{n})_{n}$ of infinite paths such that $x^{n}_{2} = x^{n+1}_{1}$ for all natural $n \geq 1$). 

\begin{proposition}\label{propgeqdelta}
Let $k \in \mathbb{N}$, $k \geq 1$, and let $(x^{n})_{n}$ be a $\frac{1}{2^{N(k)}}$-chain of infinite paths. Suppose that there is an element $x \in X$, $x = x_{1}x_{2}x_{3}\ldots$ and a natural $n \geq 0$ such that $d(\sigma^{n}(x), x^{n+1}) \geq \dfrac{1}{2^{N(k)}}$. Then $x_{n+\mathfrak{i}_{ab}} \neq x^{n+\mathfrak{i}_{ab}}_{1}$, where $a = \sigma^{n}(x)$, $b = x^{n+1}$, and $\mathfrak{i}_{ab} = \mbox{min}\{i: a_{i} \neq b_{i} \}$. 
\end{proposition}
\begin{proof}
Let $(x^{n})_{n}$ be a $\frac{1}{2^{N(k)}}$-chain and $x \in X$ be as in the hypothesis. Then, for $a= \sigma^{n}(x)$ and $b = x^{n+1}$, from Proposition~\ref{characteridelta} we obtain that $\mathfrak{i}_{ab} \leq k$ and
\begin{itemize}
    \item $\mathfrak{i}_a$ or $\mathfrak{i}_a$ is not defined, or
    \item $\mathfrak{i}_{a} > \mathfrak{i}_{ab}$ or $\mathfrak{i}_{b} > \mathfrak{i}_{ab}$.
\end{itemize}
 We deal with the second case above and leave the first one to the reader. Notice that, by the definition of $\mathfrak{i}_{ab}$, we have that $x_{n+\mathfrak{i}_{ab}} = \sigma^{n}(x)_{\mathfrak{i}_{ab}} \neq x^{n+1}_{\mathfrak{i}_{ab}}$. Now, we split the proof into two parts.

Suppose first that $\mathfrak{i}_{b} \leq \mathfrak{i}_{ab}$. Then $\mathfrak{i}_{a} > \mathfrak{i}_{ab}$ and, by the definition of $\mathfrak{i}_{a}$, we have that $x_{n+\mathfrak{i}_{ab}} =\sigma^{n}(x)_{\mathfrak{i}_{ab}} \in \{e_{\ell}: \ell \leq k\}$. Moreover, by the the minimality of $\mathfrak{i}_{ab}$ and because $\mathfrak{i}_{b} < \mathfrak{i}_{a}$, we must have $\mathfrak{i}_{b} = \mathfrak{i}_{ab}$. As $(x^{n})_{n}$ is a $\frac{1}{2^{N(k)}}$-chain, $\mathfrak{i}_{ab} \leq k$, and $x^{n+1}_{\mathfrak{i}_{ab}} \in \{e_{j}: j > k\}$, we obtain that  $x^{n+2}_{\mathfrak{i}_{ab}-1} \in \{e_{j}: j > k\}$. Inductively, we get that $x^{n+i}_{\mathfrak{i}_{ab}-i+1} \in \{e_{j}: j > k\}$ for all $1 \leq i \leq \mathfrak{i}_{ab}$. In particular, $x_{1}^{n+\mathfrak{i}_{ab}} \in \{e_{j}: j > k\}$. Then, $x_{n+\mathfrak{i}_{ab}} =\sigma^{n}(x)_{\mathfrak{i}_{ab}} \neq x_{1}^{n+\mathfrak{i}_{ab}}$ as desired.

Next, suppose that $\mathfrak{i}_{b} > \mathfrak{i}_{ab}$. Then, $x^{n+1}_{i} \in \{e_{\ell}: \ell \leq k\}$ for all $1 \leq i \leq \mathfrak{i}_{ab}$ and, as $\mathfrak{i}_{ab} \leq k$ and $(x^{n})_{n}$ is a $\frac{1}{2^{N(k)}}$-chain, we must have $x^{n+2}_{\mathfrak{i}_{ab}-1} = x^{n+1}_{\mathfrak{i}_{ab}}$. Proceeding inductively, we obtain that $x^{n+i}_{\mathfrak{i}_{ab} - i +1 } = x^{n + 1}_{\mathfrak{i}_{ab}}$ for all $1 \leq i \leq \mathfrak{i}_{ab}$. In particular, $x^{n+1}_{\mathfrak{i}_{ab}} = x^{n+\mathfrak{i}_{ab}}_{1}$. Then, $x_{n+\mathfrak{i}_{ab}} \neq x^{n+1}_{\mathfrak{i}_{ab}} = x^{n+\mathfrak{i}_{ab}}_{1}$, and the proof is finished.
\end{proof}

\begin{corollary}\label{propdfracchain}
Fix $k \in \mathbb{N}$, $k \geq 1$, and let $(x^{n})_{n}$ be a $\frac{1}{2^{N(k)}}$-chain formed by infinite paths which satisfies $x^{n}_{2} = x^{n+1}_{1}$ for all natural $n \geq 1$. Then, the element $x = x_{1}x_{2}x_{3}\ldots$ given by $x_{n} = x^{n}_{1}$, for all $n \geq 1$, is well defined and $\frac{1}{2^{N(k)}}$-shadows $(x^{n})_{n}$.
\end{corollary}
\begin{proof}
Let $(x^{n})_{n}$ be a $\frac{1}{2^{N(k)}}$-chain and $x$ as in the hypothesis. Since $x^{n}_{2} = x^{n+1}_{1}$ for all $n \geq 1$, we have that $s(x^{n+1}_{1}) \in r(x^{n}_{1})$ for all $n \geq 1$. This is enough to conclude that $x \in X$.

Next, suppose that $x$ does not $\frac{1}{2^{N(k)}}$-shadows $(x^{n})_{n}$. Then, there is a natural $n \geq 0$ such that $d(\sigma^{n}(x),x^{n+1}) \geq \dfrac{1}{2^{N(k)}}$. Therefore, for $a= \sigma^{n}(x)$ and $b = x^{n+1}$, Proposition~\ref{propgeqdelta} implies that $x_{n+\mathfrak{i}_{ab}} \neq x^{n+\mathfrak{i}_{ab}}_{1}$, which is a contradiction with the definition of the element $x \in X$. This finishes the proof.
\end{proof}

In the next corollary, we show that the Deaconu-Renault system associated with a finite graph always has the shadowing property. This, of course, can be obtained from the classical results in symbolic dynamics, but we include it here for completeness and to illustrate the use of the corollary above.

\begin{corollary}\label{finiteedges}
Let $E$ be a finite graph. Then, the associated Deaconu-Renault system $(X,\sigma)$ has the shadowing property.
\end{corollary}
\begin{proof}
Let $\varepsilon > 0$ and define $k_{0} = \mbox{max}\{k: e_{k} \ \mbox{is an edge of $E$}\}$ and $\delta:= \dfrac{1}{2^{N(k_{0})}}$ (see Definition \ref{pnk}). Notice that if $x,y \in X$ are infinite paths such that $d(x,y) < \delta$, then $x_{1} = y_{1}$. In fact, as the set of all edges of $E$ is $\{e_{j}: j \leq k_{0}\}$ and $d(x,y) < \delta$, by Proposition~\ref{characteridelta}, we obtain that $\mathfrak{i}_{xy} > k_{0} \geq 1$, where $\mathfrak{i}_{xy} = \mbox{min}\{i: x_{i} \neq y_{i}\}$.

Now, let $(x^{n})_{n}$ be a $\delta$-chain. As $d(\sigma(x^{n}), x^{n+1}) < \delta$ for all $n \geq 1$, we have that $x_{2}^{n} = x_{1}^{n+1}$ for all $n \geq 1$. Then, from Corollary~\ref{propdfracchain}, we obtain that the infinite path $x:= x^{1}_{1}x^{2}_{1}x^{3}_{1}\dots$ $\delta$-shadows $(x^{n})_{n}$.
\end{proof}

Before we proceed to define the First Infinite Path Condition, we need the following lemma, which is similar to Lemma~\ref{remoafogado}.

\begin{lemma}\label{f(in)infinite}
Let $\{\lambda^{1}, \lambda^{2}, \lambda^{3}, \ldots\}$ be an infinite set of finite paths such that $|\lambda^{i}| \geq 2$ for all $i \in \mathbb{N}$, and consider the function: 
\begin{center}
    $f:\bigcup_{i=1}^\infty\{(i,n)\in \N\times \N: \ 1 \leq n < |\lambda^{i}|\} \rightarrow \mathbb{N}$\\
    $(i,n)\mapsto \displaystyle\left(\sum_{j=0}^{i-1}|\lambda^{j}|\right) + n -i +1$,
\end{center}
where $|\lambda^{0}| := 0$. Then, $f$ is a bijection.
\end{lemma}
\begin{proof}
We omit the proof because it is analogous to the proof of Lemma~\ref{remoafogado}.
\end{proof}

In the definition of the First Infinite Path Condition below, we use the notion of $F_\varepsilon$ as given in Definition~\ref{Fepsilon}.

\begin{definition}\label{FIPC}\textbf{(First Infinite Path Condition - IPC1)} Let $(X,\sigma)$ be the Deaconu-Renault system associated with a graph. We say that $(X,\sigma)$ satisfies the First Infinite Path Condition if, for all $\varepsilon > 0$, there is $0 < \delta \leq \varepsilon$ such that for all infinite set $\{\lambda^{1}, \lambda^{2}, \ldots\}$ of finite paths which satisfy:
\begin{enumerate}
    \item $|\lambda^{i}| \geq 2$, for all $i \in \mathbb{N}$;
    \item $(F_{\delta})^{c} \cap [\{\lambda^{i}_{j}: \  i \in \mathbb{N}, \ 1 \leq j \leq |\lambda^{i}| \}-\{\lambda^{1}_{1}\}]= \{\lambda^{i+1}_{1}: i \in \mathbb{N}\} \cup \{\lambda^{i}_{|\lambda^{i}|}: i \in \mathbb{N}\}$,
    \end{enumerate}
   there is an infinite path $\lambda$ such that, for all $i,n$ for which $(i,n)\in Dom(f)$ (see Definition~\ref{f(in)infinite}), the following conditions are satisfied:
\begin{enumerate}[a.]
        \item If $\lambda_{n}^{i} \in F_{\varepsilon}$, then $\lambda_{f(i,n)} = \lambda^{i}_{n}$;
        \item If $\lambda_{n}^{i} \notin F_{\varepsilon}$, then $\lambda_{f(i,n)} \notin F_{\varepsilon}$.
\end{enumerate}
\end{definition}

Before we define the Second Infinite Path Condition, we prove an auxiliary result that characterizes infinite paths that shadow certain chains.

\begin{lemma}\label{lambda1lambda2cdots}
Given $\varepsilon>0$, let $0 < \delta \leq \varepsilon$ arise as in Definition~\ref{FIPC} and let $\{\lambda^{1}, \lambda^{2}, \lambda^{3} \ldots\}$ be an infinite set of finite paths which satisfy Conditions~1. and 2. of Definition~\ref{FIPC}. Moreover, for each $i \in \mathbb{N}$, let $y^{i}$ be an infinite path that extends $\lambda^i$, i.e. such that $y^{i}_{k} = \lambda^{i}_{k}$ for all $1 \leq k \leq |\lambda^{i}|$. With the function $f$ of Lemma~\ref{f(in)infinite} in mind, define the following chain: 
\begin{equation}\label{xinfiny}
    (x^{f(i,n)})_{f(i,n)=1}^{+ \infty}, \ \mbox{where \ } x^{f(i,n)} := \sigma^{n-1}(y^{i}).
\end{equation} 
Then, $(x^{f(i,n)})_{f(i,n)=1}^{+\infty}$ is an infinite $\delta$-chain and every infinite path $x \in X$  which $\varepsilon$-shadows such chain satisfies, for all $i,n$ for which $(i,n)\in Dom(f)$:
\begin{enumerate}[a.]
    \item If $\lambda_{n}^{i} \in F_{\varepsilon}$, then $x_{f(i,n)} = \lambda^{i}_{n}$;
        \item If $\lambda_{n}^{i} \notin F_{\varepsilon}$, then $x_{f(i,n)} \notin F_{\varepsilon}$. 
\end{enumerate}
\end{lemma}
\begin{proof}
Most of the proof is similar to the proof of Lemma~\ref{lemmafinfinite} and so we provide fewer details here. 

We start by proving that $(x^{f(i,n)})_{f(i,n)=1}^{+\infty}$ is a $\delta$-chain, and for this we split the proof in two cases. First, fix $i \in \mathbb{N}$ and suppose that $1 \leq n < |\lambda^{i}|-1$. In this case, $f(i,n)+1 = f(i,n+1)$ and hence, $$\sigma(x^{f(i,n)})= \sigma(\sigma^{n-1}(y^{i})) = \sigma^{n}(y^{i})=x^{f(i, n+1)}=x^{f(i,n)+1},$$
which means that $d(\sigma(x^{f(i,n)}),x^{f(i,n)+1}) = 0 < \delta$.
Next, fix $i \in \mathbb{N}$ and suppose that $n=|\lambda^{i}|-1$. In this case, we have that $f(i,n)+1 = f(i+1,1)$. Then, $$\sigma(x^{f(i,n)})_{1} = \sigma(\sigma^{|\lambda^{i}|-2}(y^{i}))_{1} = \sigma^{|\lambda^{i}|-1}(y^{i})_{1} = y^{i}_{|\lambda^{i}|} \notin F_{\delta}$$ and also $x^{f(i,n)+1}_{1} = x^{f(i+1,1)}_{1} \notin F_{\delta}$. Therefore, we again obtain that $d(\sigma(x^{f(i,n)}), x^{f(i,n)+1}) < \delta$ and we conclude that $(x^{f(i,n)})_{f(i,n)=1}^{\infty}$ is a $\delta$-chain.

To prove the last part of the lemma, notice first that by Equality~(\ref{xinfiny}) we have that $x^{f(i,n)}_{1} = y^{i}_{n} = \lambda^{i}_{n}$, for all $(i,n) \in Dom(f)$. Now, let $x \in X$ be an infinite path which $\varepsilon$-shadows the infinite chain $(x^{f(i,n)})_{f(i,n)=1}^{\infty}$. Suppose that $\lambda_{n}^{i} \in F_{\varepsilon}$. Then, $x^{f(i,n)}_{1} \in F_{\varepsilon}$ and so we must have that $\sigma^{f(i,n)-1}(x)_{1} = x_{1}^{f(i,n)}$ (because $d(\sigma^{f(i,n)-1}(x),x^{f(i,n)}) < \varepsilon$). Since $\sigma^{f(i,n)-1}(x)_{1} = x_{f(i,n)}$, we obtain that  Condition a. of the lemma is satisfied. To finish, suppose that $\lambda^{i}_{n} \notin F_{\varepsilon}$. Then, $x^{f(i,n)}_{1} \notin F_{\varepsilon}$ and hence we must have that $\sigma^{f(i,n)-1}(x)_{1} \notin F_{\varepsilon}$ (because $d(\sigma^{f(i,n)-1}(x),x^{f(i,n)}) < \varepsilon$). Since $\sigma^{f(i,n)-1}(x)_{1} = x_{f(i,n)}$, we obtain that Condition~b. of the lemma is also satisfied, which ends the proof.   
\end{proof}

Next, we introduce the second condition necessary to characterize infinite shadowing of Deaconu-Renault systems associated with graphs.

\begin{definition}\label{SIPC}\textbf{(Second Infinite Path Condition - IPC2)} Let $(X,\sigma)$ be the Deaconu-Renault system associated with a graph. We say that $X$ satisfies the Second Infinite Path Condition if, for all $\varepsilon > 0$, there is $0 < \delta \leq \varepsilon$ such that for all finite set of finite paths $\{\lambda^{1}, \lambda^{2}, \ldots, \lambda^{\ell}\}$ and all infinite path $\gamma$ which satisfy:
\begin{enumerate}
    \item $|\lambda^{i}| \geq 2$ for all $1 \leq i \leq 2$;
    \item $(F_{\delta})^{c} \cap [\{\lambda^{i}_{j}: \  1 \leq i\leq \ell, \ 1 \leq j \leq |\lambda^{i}| \}-\{\lambda^{1}_{1}\}]= \{\lambda^{i+1}_{1}: 1 \leq i < \ell\} \cup \{\lambda^{i}_{|\lambda^{i}|}: 1 \leq i \leq \ell\}$;
    \item $\gamma_{1} \notin F_{\delta}$;
    \end{enumerate}
   there is an infinite path $\lambda$ such that, considering the function $f$ as in Definition \ref{f(in)}, has the following properties:
\begin{enumerate}[a.]
        \item If $1 \leq i \leq \ell$, $1 \leq n < |\lambda^{i}|$, and $\lambda_{n}^{i} \in F_{\varepsilon}$, then $\lambda_{f(i,n)} = \lambda^{i}_{n}$;
        \item If $1 \leq i \leq \ell$, $1 \leq n < |\lambda^{i}|$, and $\lambda_{n}^{i} \notin F_{\varepsilon}$, then $\lambda_{f(i,n)} \notin F_{\varepsilon}$;
        \item If $s \in \mathbb{N}$ and $\gamma_{s} \in F_{\varepsilon}$, then $\lambda_{f(\ell, |\lambda^{\ell}|-1)+s} = \gamma_{s}$;
        \item If  $s \in \mathbb{N}$ and $\gamma_{s} \notin F_{\varepsilon}$, then $\lambda_{f(\ell, |\lambda^{\ell}|-1)+s} \notin F_{\varepsilon}$.
\end{enumerate}
\end{definition}

We need one last auxiliary result before giving the characterization of shadowing in terms of the first and second infinite path conditions.

\begin{lemma}\label{lambdagamma}
Given $\varepsilon>0$, let $0 < \delta \leq \varepsilon$ arise as in Definition \ref{SIPC}. Let $\{\lambda^{1}, \lambda^{2}, \lambda^{3}, \ldots, \lambda^{\ell}\}$ be a finite set of finite paths and $\gamma$ be an infinite path that satisfy Conditions~1., 2., and 3. of Definition~\ref{SIPC}. Moreover,  for each $1 \leq i \leq |\lambda^{\ell}|$, let $y^{i}$ be an infinite path that extends $\lambda^i$, i.e. such that $y^{i}_{k} = \lambda^{i}_{k}$ for all $1 \leq k \leq |\lambda^{i}|$. With the function $f$ of Definition \ref{f(in)} in mind,  define the following chain: 
\begin{equation}\label{xfininfi}
    (x^{j})_{j=1}^{+ \infty}, \  \text{where} \ x^{j}= \left\{
		\begin{array}{ll}
		\sigma^{n-1}(y^{i}), & \text{if $j=f(i,n)$, $1 \leq i \leq \ell$ and $1 \leq n < |\lambda^{i}|$};\\
		\sigma^{s-1}(\gamma), & \text{if $j=f(\ell,|\lambda^{\ell}|-1) +s$ and $s \in \mathbb{N}$}.\\
				\end{array}
		\right.
\end{equation} 
Then, $(x^{j})_{j=1}^{+\infty}$ is an infinite $\delta$-chain and every infinite path $x \in X$  which $\varepsilon$-shadows such chain satisfies:

\begin{enumerate}[a.]
        \item If $1 \leq i \leq \ell$, $1 \leq n < |\lambda^{i}|$, and $\lambda_{n}^{i} \in F_{\varepsilon}$, then $x_{f(i,n)} = \lambda^{i}_{n}$;
        \item If $1 \leq i \leq \ell$, $1 \leq n < |\lambda^{i}|$, and $\lambda_{n}^{i} \notin F_{\varepsilon}$, then $x_{f(i,n)} \notin F_{\varepsilon}$;
        \item If $s \in \mathbb{N}$ and $\gamma_{s} \in F_{\varepsilon}$, then $x_{f(\ell, |\lambda^{\ell}|-1)+s} = \gamma_{s}$;
        \item If $s \in \mathbb{N}$ and $\gamma_{s} \notin F_{\varepsilon}$, then $x_{f(\ell, |\lambda^{\ell}|-1)+s} \notin F_{\varepsilon}$.
\end{enumerate}
\end{lemma}
\begin{proof}

As part of the proof is similar to what we have done in Lemma~\ref{lemmafinfinite} and in Lemma \ref{lambda1lambda2cdots}, we will omit some of its details. 

We start by proving that $(x^{j})_{j=1}^{+\infty}$ is a $\delta$-chain, and for this we split this part of the proof into three cases. First, consider $j=f(i,n)$, $ 1 \leq i \leq \ell$ and suppose that $1 \leq n < |\lambda^{i}|-1$. In this case, $f(i,n)+1 = f(i,n+1)$. Then, $$\sigma(x^{f(i,n)})= \sigma(\sigma^{n-1}(y^{i})) = \sigma^{n}(y^{i})=x^{f(i, n+1)}=x^{f(i,n)+1},$$
which means that $d(\sigma(x^{f(i,n)}),x^{f(i,n)+1}) = 0 < \delta$.
Next, consider again $j=f(i,n)$ and $ 1 \leq i \leq \ell$, but now suppose that $n=|\lambda^{i}|-1$. In this case, we have that $f(i,n)+1 = f(i+1,1)$. Then, $$\sigma(x^{f(i,n)})_{1} = \sigma(\sigma^{|\lambda^{i}|-2}(y^{i}))_{1} = \sigma^{|\lambda^{i}|-1}(y^{i})_{1} = y^{i}_{|\lambda^{i}|} \notin F_{\delta}$$ and also $x^{f(i,n)+1}_{1} = x^{f(i+1,1)}_{1} \notin F_{\delta}$. Hence, we have again that $d(\sigma(x^{f(i,n)}), x^{f(i,n)+1}) < \delta$. 
In the third and final case, we suppose that $j=f(\ell, |\lambda^{\ell}|-1)+s$, for some $s \in \mathbb{N}$. By direct computation, we obtain that $\sigma(x^{j}) = \sigma(\sigma^{s-1}(\gamma)) = \sigma^{s}(\gamma) = x^{j+1}$. Then, $d(\sigma(x^{j}), x^{j+1}) = 0 < \delta$ and we conclude that  $(x^{j})_{j=1}^{+\infty}$ is a $\delta$-chain, as desired.

To prove the last part of the lemma, let $x \in X$ be an infinite path which $\varepsilon$-shadows the infinite chain $(x^{j})_{j=1}^{+\infty}$ (i.e. such that $d(\sigma^{j-1}(x),x^{j}) < \varepsilon$ for all $j \geq 1$). Notice that, by Equality~(\ref{xfininfi}), we have that: 
\begin{equation}\label{equalityinfinitegamma}
     x^{j}_{1}= \left\{
		\begin{array}{ll}
		y^{i}_{n}, & \text{if $j=f(i,n)$, $1 \leq i \leq \ell$ and $1 \leq n < |\lambda^{i}|$};\\
		\gamma_{s}, & \text{if $j=f(\ell,|\lambda^{\ell}|-1) +s$ and $s \in \mathbb{N}$}.\\
				\end{array}
		\right.
\end{equation} 
 First, let 
  $1 \leq i \leq \ell$, $1 \leq n < |\lambda^{i}|$, and suppose that $\lambda_{n}^{i} \in F_{\varepsilon}$.  Then, $x^{f(i,n)}_{1} = y^{n}_{i} = \lambda_{n}^{i}  \in F_{\varepsilon}$ and so we must have that $\sigma^{f(i,n)-1}(x)_{1} = x_{1}^{f(i,n)}$ (because $d(\sigma^{f(i,n)-1}(x),x^{f(i,n)}) < \varepsilon$). As $\sigma^{f(i,n)-1}(x)_{1} = x_{f(i,n)}$, we obtain that Condition a. of the lemma is satisfied. In order to verify Condition~b., let 
  $1 \leq i \leq \ell$, $1 \leq n < |\lambda^{i}|$, and suppose that $\lambda^{i}_{n} \notin F_{\varepsilon}$. Then, $x^{f(i,n)}_{1} \notin F_{\varepsilon}$ and hence we must have that $\sigma^{f(i,n)-1}(x)_{1} \notin F_{\varepsilon}$ (because $d(\sigma^{f(i,n)-1}(x),x^{f(i,n)}) < \varepsilon$). As $\sigma^{f(i,n)-1}(x)_{1} = x_{f(i,n)}$, we conclude that Condition b. is satisfied. To verify Itens c. and d. of the lemma, let $s \in \mathbb{N}$ and $j=f(\ell, |\lambda^{\ell}| -1)+s$. Notice that, for such $j$, Equality~(\ref{equalityinfinitegamma}) says that $x^{j}_{1} = \gamma_{s}$. 
 To prove Item c., suppose that $\gamma_{s} \in F_{\varepsilon}$. Then, $x^{j}_{1} \in F_{\varepsilon}$ and, since $d(\sigma^{j-1}(x),x^{j}) < \varepsilon$, we must have that $\sigma^{j-1}(x)_{1} = x^{j}_{1}$. As $\sigma^{j-1}(x)_{1} = x_{j}$, this implies Item~c.. Finally, suppose that $\gamma_{s} \notin F_{\varepsilon}$. Then, $x^{j}_{1} \notin F_{\varepsilon}$ and, as $d(\sigma^{j-1}(x),x^{j}) < \varepsilon$, we must have that $\sigma^{j-1}(x)_{1} \notin F_{\varepsilon}$. Since $\sigma^{j-1}(x)_{1} = x_{j}$, this implies Item~d. and the proof is finished. 
\end{proof}

We can now characterize shadowing for Deaconu-Renault systems associated with graphs in terms of the First and Second Infinite Path Conditions. 

\begin{theorem}\label{teoinfiniteshadowing}
Let $(X,\sigma)$ be the Deaconou-Renault system associated with a graph. Then $(X,\sigma)$ has the shadowing property if, and only if, it satisfies the First and the Second Infinite Path Conditions  (see Definitions \ref{FIPC} and \ref{SIPC}).
\end{theorem}
\begin{proof}

Suppose that $(X,\sigma)$ has the shadowing property. Given $\varepsilon > 0$, let $0 < \delta \leq \varepsilon$ witness $\varepsilon$-shadowing.  
Let $\{\lambda^{1}, \lambda^{2}, \lambda^{3}, \ldots\}$ be an infinite set of finite paths which satisfy Conditions~1. and 2. of Definition \ref{FIPC} and, for each $i \in \mathbb{N}$, let $y^{i}$ be an infinite path such that $y^{i}_{k} = \lambda^{i}_{k}$ for all $1 \leq k \leq |\lambda^{i}|$. With the function $f$ of Lemma~\ref{f(in)infinite} in mind, consider the following infinite chain:
\begin{equation}
    (x^{f(i,n)})_{f(i,n)=1}^{+ \infty}, \  \text{where} \ x^{f(i,n)} = \sigma^{n-1}(y^{i}).
\end{equation}
By Lemma \ref{lambda1lambda2cdots}, such chain is an infinite $\delta$-chain. Let $\lambda \in X$ be an infinite path that $\varepsilon$-shadows it. Again by Lemma~\ref{lambda1lambda2cdots}, $\lambda$ satisfies Itens~a. and b. of 
Definition~\ref{FIPC}. Therefore, $X$ satisfies the First Infinite Path Condition. 

Now we move to the Second Infinite Path Condition. Let $\{\lambda^{1},\lambda^{2},\ldots,\lambda^{\ell}\}$ be a finite set of finite paths, $\gamma$ be an infinite path which satisfy conditions 1., 2., and 3. of Definition \ref{SIPC}, and, for each $1 \leq i \leq \ell$, let $y^{i}$ be an infinite path such that $y^{i}_{k} = \lambda^{i}_{k}$ for all $1 \leq k \leq |\lambda^{i}|$. Now, using the function $f$ as in Definition \ref{f(in)}, we define the following infinite chain: 
\begin{equation}
    (x^{j})_{j=1}^{+ \infty}, \  \text{given by} \ x^{j}= \left\{
		\begin{array}{ll}
		\sigma^{n-1}(y^{i}), & \text{if $j=f(i,n)$, $1 \leq i \leq \ell$ and $1 \leq n < |\lambda^{i}|$};\\
		\sigma^{s-1}(\gamma), & \text{if $j=f(\ell,|\ell|-1) +s$ and $s \in \mathbb{N}$}.\\
				\end{array}
		\right.
\end{equation}

By Lemma \ref{lambdagamma}, such a chain is an infinite $\delta$-chain. Hence, there is an infinite path $x \in X$ that $\varepsilon$-shadows such a chain. Again by Lemma~\ref{lambdagamma}, $x$ is an infinite path that satisfies Conditions~a., b., c., and d. of Definition~\ref{SIPC}. Therefore, $(X,\sigma)$ satisfies the Second Infinite Path Condition.

For the converse, suppose that $(X,\sigma)$ satisfies the First and the Second Infinite Path Conditions. Given $\varepsilon > 0$, let $\varepsilon':=\frac{1}{2^{N(k)}}$ be a real number such that $\frac{1}{2^{N(k)}}\leq \varepsilon$ and $N(k)$ is as in Definition~\ref{pnk}. Let $\delta_{1},\delta_{2} > 0$ arise from the First and the Second Infinite Path Conditions, respectively, in relation to $\varepsilon'$. Choose $\delta = \text{min}\{\delta_{1}, \delta_{2}\}$ and let $(x^{n})_{n=1}^{\infty}$ be an infinite $\delta$-chain. 

Suppose that $x_{1}^{n} \in F_{\delta}$ for all $n \geq 2$. In this case, as $(x^{n})_{n=1}^{\infty}$ is a $\delta$-chain, we must have that $\sigma(x^{n})_{1} = x^{n+1}_{1}$ for all $n \geq 1$. Since $(x^{n})_{n=1}^{\infty}$ is also a $\frac{1}{2^{N(k)}}$-chain we obtain, from Corollary~\ref{propdfracchain}, an infinite path $x \in X$ that $\frac{1}{2^{N(k)}}$-shadows $(x^{n})_{n=1}^{\infty}$. As $\frac{1}{2^{N(k)}} \leq \varepsilon$, $x$ also $\varepsilon$-shadows it.

We are left with the case in which there is an $n \geq 2$ such that $x^{n}_{1} \notin F_{\delta}$. In this case, consider the set 
\begin{equation}
A_{\delta} = \{n \geq 2: x_{1}^{n} \notin F_{\delta}\} \neq \emptyset.    
\end{equation}
We now split the proof into two arguments, depending on whether the cardinality of $A_{\delta}$ is finite.

Suppose first that $A_{\delta}$ is an infinite set, say $A_{\delta}=\{n_{1}, n_{2}, n_{3}, \ldots\}$. Notice that $\sigma(x^{n_{i}-1})_{1} \notin F_{\delta}$ for all $i \in \mathbb{N}$ (because $d\left(\sigma(x^{n_{i}-1}), x^{n_{i}}\right)< \delta$ and $x^{n_{i}}_{1} \notin F_{\delta}$). Also, if $n \geq 2$ and $n \notin A_{\delta}$, then $\sigma(x^{n-1})_{1}= x^{n}_{1}$ (because $x^{n}_{1} \in F_{\delta}$ and $d\left(\sigma(x^{n-1}),x^{n}\right) < \delta$). Given this information, for each $i \in \mathbb{N}$, define the finite path $\lambda^{i}$ by:
\begin{equation}\label{eqlambdaiN}
\lambda^{i}:=
		x^{n_{i-1}}_{1}x^{n_{i-1}+1}_{1}\ldots x^{n_{i}-1}_{1}\sigma(x^{n_{i}-1})_{1}, 
\end{equation}
where $n_{0}:=1$. Notice that:
\begin{equation}\label{|eqlambdaiN|}
|\lambda^{i}|=
		n_{i}-n_{i-1}+1,
\end{equation}
for each natural $i$. By Equalities~(\ref{eqlambdaiN}) and (\ref{|eqlambdaiN|}) we have that:
\begin{equation}\label{f(in)inf}
    f(i,n) = \displaystyle \left(\sum_{j=1}^{i-1}|\lambda^{j}|\right) + n -i + 1 = n_{i-1} + n -1. 
\end{equation}
Now, by Equalities (\ref{eqlambdaiN}) and (\ref{f(in)inf}), we have that (for $(i,n) \in \{(i,n): i \in \mathbb{N}, 1 \leq n < |\lambda^{i}|\}$):
\begin{equation}
    \lambda^{i}_{n} = x_{1}^{n_{i-1}+n-1} = x^{f(i,n)}_{1}.
\end{equation}
We will use these equalities later in the proof. For now, notice that
the finite paths $\{\lambda^{i}:i \in \mathbb{N}\}$ satisfy the first and second conditions of Definition~\ref{FIPC}. 
Hence, as $X$ satisfies the First Infinite Path Condition (Definition \ref{FIPC}), there exists an infinite path $\lambda$ such that Conditions~a. and b. of Definition~\ref{FIPC} are satisfied. 
Denoting the infinite path $\lambda$ also by $x$ (as it is more intuitive at some places), we have that:
\begin{equation}\label{xlambdaxinfinite}
    x_{f(i,n)}=\lambda_{f(i,n)}=\lambda^{i}_{n}=x_{1}^{f(i,n)}, \  \text{if $\lambda_{n}^{i} \in F_{\varepsilon'}$. }
\end{equation}
By the equation above, we have that $\lambda_{s} = x^{s}_{1}$ for all $s \in \mathbb{N}$ such that $s=f(i,n)$ and $\lambda_{n}^{i} \in F_{\varepsilon'}$.
Also, by a similar approach, we have that $\lambda_{s} \notin F_{\varepsilon'}$ if $x^{s}_{1} \notin F_{\varepsilon'}$, for all $s \in \mathbb{N}$.
Now we prove that $x$ $\varepsilon'$-shadows the $\delta$-chain $(x^{n})_{n=1}^{\infty}$ (and hence it also $\varepsilon$-shadows $(x^{n})_{n=1}^{\infty}$). Fix a positive integer $s \geq 0$, consider $a:=\sigma^{s}(x)$, and $b:=x^{s+1}$. To proceed, we split the proof into a few cases and check that $d(a,b)< \varepsilon'$ in all of them. 

If $a=b$, then clearly $d(a,b) < \varepsilon'$. If $a \neq b$, then we consider $\mathfrak{i}_{ab}= \mbox{min}\{i: a_{i} \neq b_{i}\}$. If $\mathfrak{i}_{ab} > k$, then $d(a,b) < \varepsilon'$ (see Definition \ref{pnk} and Proposition \ref{characteridelta} for further details). Suppose that $\mathfrak{i}_{ab} \leq k$. If $a_{\mathfrak{i}_{ab}}$ and $b_{\mathfrak{i}_{ab}}$ both do not belong to $F_{\varepsilon'}$, then we have that $d(a,b)< \varepsilon'$ (because any finite path which is an initial segment of $a$ or $b$, but not the other, must contain an edge which is not in $F_{\varepsilon'}$). We are left with the cases where $b_{\mathfrak{i}_{ab}} \in F_{\varepsilon'}$ or where $a_{\mathfrak{i}_{ab}} \in F_{\varepsilon'}$. In both situations, to conclude that $d(a,b) < \varepsilon'$, we must prove that there is $1 \leq i < \mathfrak{i}_{ab}$ such that $a_{i}=b_{i} \notin F_{\varepsilon'}$. We do this in the two following paragraphs.

Suppose that $x^{s+1}_{\mathfrak{i}_{ab}}=b_{\mathfrak{i}_{ab}} \in F_{\varepsilon'}$. Moreover, suppose that for all $1 \leq i < \mathfrak{i}_{ab}$ we have $a_{i}=b_{i} \in F_{\varepsilon'}$. As $(x^{n})_{n=1}^{\infty}$ is an $\varepsilon'$-chain, we must have that $x^{s+j}_{1} = x^{s+1}_{j}$, for all $1 \leq j \leq \mathfrak{i}_{ab}$. In particular, $x^{s+\mathfrak{i}_{ab}}_{1} = x^{s+1}_{\mathfrak{i}_{ab}} = b_{\mathfrak{i}_{ab}} \in F_{\varepsilon'}$. Then, $a_{\mathfrak{i}_{ab}} = x_{s+\mathfrak{i}_{ab}} \neq x^{s+\mathfrak{i}_{ab}}_{1} \in F_{\varepsilon'}$, which contradicts Equation (\ref{xlambdaxinfinite}). Therefore, there exists an index $i$, $1 \leq i < \mathfrak{i}_{ab}$, such that $a_{i}=b_{i} \notin F_{\varepsilon'}$. Hence, $d(a,b)< \varepsilon'$. 

Next, suppose that $\sigma^{s}(x)_{\mathfrak{i}_{ab}}=a_{\mathfrak{i}_{ab}}\in F_{\varepsilon'}$. We may assume that $x^{s+1}_{\mathfrak{i}_{ab}}=b_{\mathfrak{i}_{ab}} \notin F_{\varepsilon'}$ (as we have dealt with the case $b_{\mathfrak{i}_{ab}} \in F_{\varepsilon'}$ in the previous paragraph). Suppose that for all $1 \leq i < \mathfrak{i}_{ab}$ we have $a_{i}=b_{i} \in F_{\varepsilon'}$. As $(x^{n})_{n=1}^{\infty}$ is a $\varepsilon'$-chain, we must have that $x^{s+j}_{\mathfrak{i}_{ab}-j+1} \notin F_{\varepsilon'}$, for all $1 \leq j \leq \mathfrak{i}_{ab}$. In particular, $x^{s+\mathfrak{i}_{ab}}_{1} \notin F_{\varepsilon'}$. Then, 
\begin{equation}\label{wedonot}
    a_{\mathfrak{i}_{ab}} = x_{s + \mathfrak{i}_{ab}} \neq x^{s+\mathfrak{i}_{ab}}_{1}.
\end{equation} If we write $s+\mathfrak{i}_{ab}=f(i,n)$, then we have that $\lambda^{i}_{n} \in F_{\varepsilon'}$ and hence, by Equality~(\ref{xlambdaxinfinite}), we must have that $x_{f(i,n)} = x^{f(i,n)}_{1}$. But notice that, in this case, Equality~(\ref{wedonot}) contradicts Equality~(\ref{xlambdaxinfinite}). Therefore, there exists an index $i$, $1 \leq i < \mathfrak{i}_{ab}$, such that $a_{i}=b_{i} \notin F_{\varepsilon'}$, which allow us to conclude that $d(a,b)< \varepsilon'$ as desired.

To finish the proof, we must consider the case when $A_{\delta} = \{n \geq 2: x_{1}^{n} \notin F_{\delta}\}$ is a finite set, say $A_{\delta} = \{n_{1}, n_{2}, \ldots, n_{\ell}\}$. Notice that $\sigma(x^{n_{i}-1})_{1} \notin F_{\delta}$ for all $1 \leq i \leq \ell$ (because $d\left(\sigma(x^{n_{i}-1}), x^{n_{i}}\right)< \delta$ and $x^{n_{i}}_{1} \notin F_{\delta}$). Moreover, if $n \geq 2$ and $n \notin A_{\delta}$, then $\sigma(x^{n-1})_{1}= x^{n}_{1}$ (because $x^{n}_{1} \in F_{\delta}$ and $d\left(\sigma(x^{n-1}),x^{n}\right) < \delta$). Given this information, define the infinite path $\gamma$ by:
\begin{equation}\label{gammainfinite}
    \gamma_{s} := x_{1}^{n_{\ell}+s-1} = x^{f(\ell, |\lambda^{\ell}|-1) + s}_{1}, \text{for all} \ s \in \mathbb{N}
\end{equation}
and, for each $1 \leq i \leq \ell$, define the finite path $\lambda^{i}$ by:
\begin{equation}
\lambda^{i}:=
		x^{n_{i-1}}_{1}x^{n_{i-1}+1}_{1}\ldots x^{n_{i}-1}_{1}\sigma(x^{n_{i}-1})_{1},\\
\end{equation}
where $n_{0}:=1$. Notice that, for each $1 \leq i \leq \ell$, we have that $|\lambda^{i}|= n_{i}-n_{i-1}+1$ (which is similar to Equality~(\ref{eqlambdaiN})).  Furthermore, the infinite path $\gamma$ and the finite paths $\{\lambda^{1}, \lambda^{2}, \ldots, \lambda^{\ell}\}$ satisfy the Conditions~1. to 3. of Definition~\ref{SIPC}.
Then, since $X$ satisfies the Second Infinite Path Condition,  there is an infinite path $\lambda$ that satisfies itens a. to d. of Definition \ref{SIPC}.

Denoting, as before, the infinite path $\lambda$ also by $x$, we will prove that  $x:=\lambda$ is a path that $\varepsilon$-shadows the chain $(x^{n})_{n=1}^{\infty}$. 

Fix a positive integer $j \geq 0$ and consider $a:=\sigma^{j}(x)$ and $b:=x^{j+1}$. As before, we split the proof in a few cases.
 
 The case $0 \leq j \leq f(\ell,|\lambda^{\ell}|-1)-1$ is similar to the case considered when $\# A_{\delta} = \infty$ and therefore we omit its proof. We assume $j > f(\ell,|\lambda^{\ell}|-1)-1$. 
 
 If $j = f(\ell,|\lambda^{\ell}|-1)$, then we have that $x^{j+1}_{1} = x^{n_{\ell}}_{1}$. Recall that $x^{n_{\ell}}_{1} \notin F_{\delta}$. By the definition of $\gamma$ (see Equality~ \ref{gammainfinite}), we have that $\gamma_{1} = x^{n_{\ell}}_{1} \notin F_{\delta}$. Hence $\gamma_{1} \notin F_{\varepsilon'}$ and, by item d. of the Second Infinite Path Condition, we have that $\sigma^{j}(x)_{1} = x_{j+1} \notin F_{\varepsilon'}.$ Since $\sigma^{j}(x)_{1}$ and $x^{j+1}_{1}$ do not belong to $ F_{\varepsilon'}$, we have that $d(\sigma^{j}(x), x^{j+1}) < \varepsilon'$. 
 
 Finally, consider $j > f(\ell,|\lambda^{\ell}|-1)$, say $j = f(\ell,|\lambda^{\ell}|-1) + s$ for $s \in \mathbb{N}$. If $a=b$ then clearly $d(a,b) < \varepsilon'$. If $a \neq b$, then we consider $\mathfrak{i}_{ab}= \mbox{min}\{i: a_{i} \neq b_{i}\}$. If $\mathfrak{i}_{ab} > k$, then $d(a,b) < \varepsilon'$ (see Definition \ref{pnk} and Proposition \ref{characteridelta} for further details). Suppose that $\mathfrak{i}_{ab} \leq k$. If $a_{\mathfrak{i}_{ab}}$ and $b_{\mathfrak{i}_{ab}}$ both do not belong to $F_{\varepsilon'}$, then we have that $d(a,b)< \varepsilon'$ (because any finite path which is an initial segment of $a$ or $b$, but not the other, must contain an edge which is not in $F_{\varepsilon'}$). We need to check what happens when $b_{\mathfrak{i}_{ab}} \in F_{\varepsilon'}$ and when $a_{\mathfrak{i}_{ab}} \in F_{\varepsilon'}$. As before, in both situations, we must prove that there is $1 \leq i < \mathfrak{i}_{ab}$ such that $a_{i}=b_{i} \notin F_{\varepsilon'}$.
 
 Assume that $x^{j+1}_{\mathfrak{i}_{ab}}=b_{\mathfrak{i}_{ab}} \in F_{\varepsilon'}$. Furthermore,  suppose that for all $1 \leq i < \mathfrak{i}_{ab}$ we have $a_{i}=b_{i} \in F_{\varepsilon'}$. As $(x^{n})_{n=1}^{\infty}$ is a $\varepsilon'$-chain, we must have that $x^{j+r}_{1} = x^{j+1}_{r}$ for all $1 \leq r \leq \mathfrak{i}_{ab}$. In particular, $x^{j+\mathfrak{i}_{ab}}_{1} = x^{j+1}_{\mathfrak{i}_{ab}} = b_{\mathfrak{i}_{ab}} \in F_{\varepsilon'}$. Then, $a_{\mathfrak{i}_{ab}} = x_{j+\mathfrak{i}_{ab}} \neq x^{j+\mathfrak{i}_{ab}}_{1} \in F_{\varepsilon'}$. By Equality (\ref{gammainfinite}), we have that $\gamma_{s + \mathfrak{i}_{ab}} = x^{j+\mathfrak{i}_{ab}}_{1}$. Hence $\gamma_{s + \mathfrak{i}_{ab}} \in F_{\varepsilon'}$, but $x_{f(\ell,|\lambda^{\ell}|-1) + s + \mathfrak{i}_{ab}} \neq \gamma_{s + \mathfrak{i}_{ab}}$, which contradicts item c. of the Second Infinite Path Condition. Therefore, there exists an index $i$, $1 \leq i < \mathfrak{i}_{ab}$, such that $a_{i}=b_{i} \notin F_{\varepsilon'}$, which implies that $d(a,b)< \varepsilon'$.
 
 For the final case, assume that $\sigma^{j}(x)_{\mathfrak{i}_{ab}}=a_{\mathfrak{i}_{ab}}\in F_{\varepsilon'}$. We may assume that  $x^{j+1}_{\mathfrak{i}_{ab}}=b_{\mathfrak{i}_{ab}} \notin F_{\varepsilon'}$ (because we have dealt with the case $b_{\mathfrak{i}_{ab}} \in F_{\varepsilon'}$ in the previous paragraph). Suppose that for all $1 \leq i < \mathfrak{i}_{ab}$ we have $a_{i}=b_{i} \in F_{\varepsilon'}$. As $(x^{n})_{n=1}^{\infty}$ is a $\varepsilon'$-chain, we must have that $x^{j+r}_{\mathfrak{i}_{ab}-r+1} \notin F_{\varepsilon'}$ for all $1 \leq r \leq \mathfrak{i}_{ab}$. In particular, $x^{j+\mathfrak{i}_{ab}}_{1} \notin F_{\varepsilon'}$. As $\gamma_{s+\mathfrak{i}_{ab}} = x^{j+\mathfrak{i}_{ab}}_{1}$, item d. of the Second Infinite Path Condition implies that $x_{j+\mathfrak{i}_{ab}} \notin F_{\varepsilon'}$, what contradicts the assumption at the beginning of this paragraph. Therefore, there exists an index $i$, $1 \leq i < \mathfrak{i}_{ab}$, such that $a_{i}=b_{i} \notin F_{\varepsilon'}$, and hence $d(a,b)< \varepsilon'$ and the proof is finished.
\end{proof}

\section{Examples of shadowing for graphs}\label{lagarta}

In this section, we apply our characterization of shadowing for a number of examples. In fact, we single out classes of graphs that present the shadowing property and classes which do not.

\subsection{The rose with infinite petals.} 

The rose with infinite petals is the graph that consists of only one vertex and an infinite countable number of edges, see the picture below. This graph is associated with the famous Cuntz algebra $O_\infty$. We show below that the associated Deaconu-Renault system has the shadowing property.

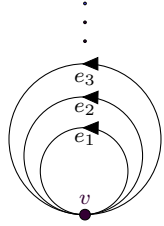
\begin{center} 
\definecolor{ududff}{rgb}{0.30196078431372547,0.30196078431372547,1}
\definecolor{ttqqtt}{rgb}{0.2,0,0.2}
\begin{tikzpicture}[line cap=round,line join=round,>=triangle 45,x=1cm,y=1cm]
\draw [line width=0.4pt] (-6,2) circle (1cm);
\draw [line width=0.4pt] (-6.0225,1.78) circle (0.7803244517506808cm);
\draw [line width=0.4pt] (-6.0114069828035435,1.5765346534653466) circle (0.5766474884217285cm);
\draw [->,line width=0.4pt] (-5.961781599302772,2.1510428320106905) -- (-6.039620787712179,2.1524915156656896);
\draw [->,line width=0.4pt] (-5.958366678740344,2.5576844907185734) -- (-6.042407920026582,2.560070461381673);
\draw [->,line width=0.4pt] (-5.966902636245555,2.999452132176677) -- (-6.043190021331042,2.999066875668203);
\begin{scriptsize}
\draw[color=black] (-6,2.8) node {$e_{3}$};
\draw [fill=ttqqtt] (-6,1) circle (2pt);
\draw[color=ttqqtt] (-6,1.2) node {$v$};
\draw[color=black] (-6,2.4) node {$e_{2}$};
\draw[color=black] (-6,2) node {$e_{1}$};
\draw [fill=ududff] (-6,3.8) circle (0.5pt);
\draw [fill=ttqqtt] (-6,3.55) circle (0.5pt);
\draw [fill=ttqqtt] (-6,3.3) circle (0.5pt);
\end{scriptsize}
\end{tikzpicture}
\end{center}
\begin{proposition}\label{coroalledges}
Let $(X,\sigma)$ be the Deaconu-Renault system associated with the rose with infinite petals. Then, $(X,\sigma)$ has the shadowing property.
\end{proposition}
\begin{proof}

 We show that $X$ satisfies the First and Second Infinite Path Conditions (Definitions~\ref{FIPC} and \ref{SIPC}), which are equivalent to the shadowing property (Theorem \ref{teoinfiniteshadowing}). 

Given $\varepsilon > 0$, pick $\delta = \varepsilon$.  Let $\{\lambda^{1}, \lambda^{2}, \ldots\}$ be an infinite set of finite paths which satisfy Conditions 1. and 2. of Definition \ref{FIPC}, which are: 
\begin{enumerate}
    \item $|\lambda^{i}| \geq 2$, for all $i \in \mathbb{N}$;
    \item $(F_{\delta})^{c} \cap [\{\lambda^{i}_{j}: \  i \in \mathbb{N}, \ 1 \leq j \leq |\lambda^{i}| \}-\{\lambda^{1}_{1}\}]= \{\lambda^{i+1}_{1}: i \in \mathbb{N}\} \cup \{\lambda^{i}_{|\lambda^{i}|}: i \in \mathbb{N}\}$.
    \end{enumerate}

Notice that the infinite path $\lambda$ given by $\lambda_{j} := \lambda^{i}_{n}$, where $(i,n)$ is the unique pair such that $1 \leq n < |\lambda^{i}|$ and $f(i,n)=j$ (see Lemma~\ref{f(in)infinite} and Definition~\ref{f(in)} to recall the function $f$) is well defined. It is immediate that if $\lambda^{i}_{n} \in F_{\varepsilon}$ then $\lambda_{f(i,n)} = \lambda^{i}_{n}$, and if $\lambda^{i}_{n} \notin F_{\varepsilon}$ then $\lambda_{f(i,n)} \notin F_{\varepsilon}$. We conclude that  $(X,\sigma)$ satisfies the First Infinite Path Condition. The proof that $(X,\sigma)$ satisfies the Second Infinite Path Condition is analogous and we leave it to the reader.
\end{proof}


\subsection{Wandering graphs} 

In this section, we identify a class of graphs, which we call wandering graphs, which present the shadowing property.

\begin{definition}\label{defiwandering}
An infinite graph $E$ is said to be a wandering graph if it satisfies the following condition (which we call Wandering Condition):\\
For all $\varepsilon > 0$ there is $\delta > 0$ such that, if $\gamma$ is a finite path such that $\gamma_{1} \notin F_{\delta}$, then $\gamma_{|\gamma|} \notin F_{\varepsilon}$.    
\end{definition}
\begin{remark}
    If $E$ is a wandering graph and $\lambda$ is an infinite path such that $\lambda_{1} \notin F_{\delta}$, then $\lambda_{j} \notin F_{\varepsilon}$ for all natural $j > 1$.
\end{remark}

Below we present a simple example of an infinite wandering graph.

\begin{example}\label{examplewandering} Let $E$ be the graph with a countable set of vertices, say $\{u_i:i\geq 1\}$, and edges $e_i$ such that $s(e_i)=u_i$ and $r(e_i)=u_{i+1}$, $i\geq 1$. Clearly, this graph is wandering. We present a picture of it below.

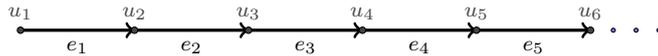
\begin{center}
\definecolor{xdxdff}{rgb}{0.49019607843137253,0.49019607843137253,1}
\definecolor{uuuuuu}{rgb}{0.26666666666666666,0.26666666666666666,0.26666666666666666}
\begin{tikzpicture}[scale=1.5][line cap=round,line join=round,>=triangle 45,x=1cm,y=1cm]
\draw [->,line width=1.2pt] (0,0) -- (1,0);
\draw [->,line width=1.2pt] (1,0) -- (2,0);
\draw [->,line width=1.2pt] (2,0) -- (3,0);
\draw [->,line width=1.2pt] (3,0) -- (4,0);
\draw [->,line width=1.2pt] (4,0) -- (5,0);
\begin{scriptsize}
\draw [fill=uuuuuu] (0,0) circle (0.8pt);
\draw[color=uuuuuu] (0,0.15) node {$u_{1}$};
\draw [fill=uuuuuu] (1,0) circle (0.8pt);
\draw[color=uuuuuu] (1,0.15) node {$u_{2}$};
\draw[color=black] (0.5,-0.15) node {$e_{1}$};
\draw [fill=uuuuuu] (2,0) circle (0.8pt);
\draw[color=uuuuuu] (2,0.15) node {$u_{3}$};
\draw[color=black] (1.5,-0.15) node {$e_{2}$};
\draw [fill=uuuuuu] (3,0) circle (0.8pt);
\draw[color=uuuuuu] (3,0.15) node {$u_{4}$};
\draw[color=black] (2.5,-0.15) node {$e_{3}$};
\draw [fill=uuuuuu] (4,0) circle (0.8pt);
\draw[color=uuuuuu] (4,0.15) node {$u_{5}$};
\draw[color=black] (3.5,-0.15) node {$e_{4}$};
\draw [fill=uuuuuu] (5,0) circle (0.8pt);
\draw[color=uuuuuu] (5,0.15) node {$u_{6}$};
\draw[color=black] (4.5,-0.15) node {$e_{5}$};
\draw [fill=xdxdff] (5.2,0) circle (0.5pt);
\draw [fill=xdxdff] (5.4,0) circle (0.5pt);
\draw [fill=xdxdff] (5.6,0) circle (0.5pt);
\end{scriptsize}
\end{tikzpicture}
\end{center}    
\end{example}

Next, we prove that the Deaconu-Renault system associated with a wandering graph always presents the shadowing property.

\begin{proposition}\label{propwandering}
Let $E$ be an infinite graph that satisfies the Wandering Condition
and let $(X,\sigma)$ be the associated Deaconu-Renault system. Then, $(X,\sigma)$ has the shadowing property.   
\end{proposition}
\begin{proof}

We show that $X$ satisfies the First Infinite Path Condition (Definition \ref{FIPC}) and leave the proof that it satisfies the Second Infinite Path Condition (Definition \ref{SIPC}) to the reader, as it is analogous to what we do below.

Given $\varepsilon>0$, Let $\delta$ be the number associated with $\varepsilon$ in the Wandering Condition. 

Notice first that, as $E$ satisfies the Wandering Condition, if a finite path $\gamma$ is such that $\gamma_{1} \notin F_{\delta}$, then $\gamma_{j} \notin F_{\varepsilon}$ for all $1 \leq j \leq |\gamma|$. Now, let $\{\lambda^{1}, \lambda^{2}, \ldots\}$ be an infinite set of finite paths which satisfy conditions 1. and 2. of the First Infinite Path Condition. Since $\lambda^{i+1}_{1} \notin F_{\delta}$ for all $i \geq 1$, we must have that $\lambda^{i+1}_{j} \notin F_{\varepsilon}$ for all $1 < j \leq |\lambda^{i+1}|$. Then, any infinite path $\lambda$ which satisfies $\lambda_{i} = \lambda^{1}_{i}$, for all $1 \leq i \leq |\lambda^{1}|$, and $\lambda_{j} \notin F_{\varepsilon}$, for all $j > |\lambda^{1}|$, satisfies Conditions~a. and b. of Definition~\ref{FIPC}. The existence of such a path follows from the fact that $E$ has no sinks and satisfies the Wandering Condition. 

\end{proof}

\subsection{Attractor graphs} 

In this section, we completely characterize the shadowing property for graphs that contain an attractor subgraph (see the definition below for the notion of an attractor subgraph).

\begin{definition}\label{attractor}
Let $E$ be a graph and $(X,\sigma)$ be the associated Deaconu-Renault system. A subgraph $E' \subset E$ is an attractor subgraph (of $E$) if for all $x \in X$, there is an $n \in \mathbb{N}$ (which depends on $x$) such that $x_{n} \in E'^{1}$.    
\end{definition}

\begin{example}\label{renewal} Below we picture the infinite renewal shift, see \cite{BisEx} for more details. This graph contains an attractor subgraph, for example, the subgraph $\{\{u_1,u_2\},\{e_2\}\}$.
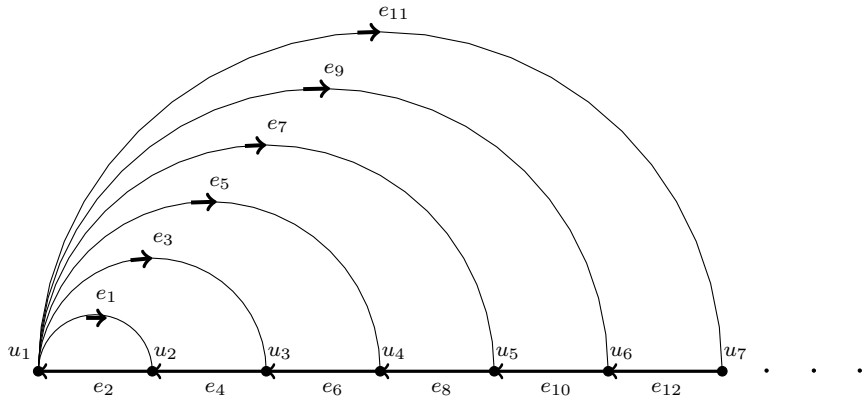
\begin{center}
\begin{tikzpicture}[scale=3][line cap=round,line join=round,>=triangle 45,x=1cm,y=1cm]
\draw [->,line width=1.2pt] (1.5,0.5) -- (1,0.5);
\draw [->,line width=1.2pt] (2,0.5) -- (1.5,0.5);
\draw [->,line width=1.2pt] (2.5,0.5) -- (2,0.5);
\draw [->,line width=1.2pt] (3,0.5) -- (2.5,0.5);
\draw [->,line width=1.2pt] (3.5,0.5) -- (3,0.5);
\draw [->,line width=1.2pt] (4,0.5) -- (3.5,0.5);
\draw [shift={(1.25,0.5)},line width=0.4pt]  plot[domain=0:3.141592653589793,variable=\t]({1*0.25*cos(\t r)+0*0.25*sin(\t r)},{0*0.25*cos(\t r)+1*0.25*sin(\t r)});
\draw [shift={(1.5,0.5)},line width=0.4pt]  plot[domain=0:3.141592653589793,variable=\t]({1*0.5*cos(\t r)+0*0.5*sin(\t r)},{0*0.5*cos(\t r)+1*0.5*sin(\t r)});
\draw [shift={(1.75,0.5)},line width=0.4pt]  plot[domain=0:3.141592653589793,variable=\t]({1*0.75*cos(\t r)+0*0.75*sin(\t r)},{0*0.75*cos(\t r)+1*0.75*sin(\t r)});
\draw [shift={(2,0.5)},line width=0.4pt]  plot[domain=0:3.141592653589793,variable=\t]({1*1*cos(\t r)+0*1*sin(\t r)},{0*1*cos(\t r)+1*1*sin(\t r)});
\draw [shift={(2.25,0.5)},line width=0.4pt]  plot[domain=0:3.141592653589793,variable=\t]({1*1.25*cos(\t r)+0*1.25*sin(\t r)},{0*1.25*cos(\t r)+1*1.25*sin(\t r)});
\draw [shift={(2.5,0.5)},line width=0.4pt]  plot[domain=0:3.141592653589793,variable=\t]({1*1.5*cos(\t r)+0*1.5*sin(\t r)},{0*1.5*cos(\t r)+1*1.5*sin(\t r)});
\draw [->,line width=1.5pt] (1.21,0.7358) -- (1.3,0.7358);
\draw [->,line width=1.5pt] (1.6700305455981517,1.24572440375964) -- (1.7825617996533436,1.249292819399289);
\draw [->,line width=1.5pt] (1.9063977935529601,1.495609676001718) -- (2,1.5);
\draw [->,line width=1.5pt] (2.162008203816862,1.746899131367275) -- (2.2800747884935095,1.7496381504647938);
\draw [->,line width=1.5pt] (2.399767163413584,1.9966473794684707) -- (2.5,2);
\draw [->,line width=1.5pt] (1.4046563131651248,0.9908254082468972) -- (1.5,1);
\begin{scriptsize}
\draw [fill=black] (1.5,0.5) circle (0.6pt);
\draw[color=black] (1.56,0.58) node {$u_{2}$};
\draw [fill=black] (1,0.5) circle (0.6pt);
\draw[color=black] (0.92,0.58) node {$u_{1}$};
\draw[color=black] (1.288641975308644,0.42) node {$e_{2}$};
\draw [fill=black] (2,0.5) circle (0.6pt);
\draw[color=black] (2.06,0.58) node {$u_{3}$};
\draw[color=black] (1.7785185185185217,0.42) node {$e_{4}$};
\draw [fill=black] (2.5,0.5) circle (0.6pt);
\draw[color=black] (2.56,0.58) node {$u_{4}$};
\draw[color=black] (2.2920987654321032,0.42) node {$e_{6}$};
\draw [fill=black] (3,0.5) circle (0.6pt);
\draw[color=black] (3.06,0.58) node {$u_{5}$};
\draw[color=black] (2.770123456790129,0.42) node {$e_{8}$};
\draw [fill=black] (3.5,0.5) circle (0.6pt);
\draw[color=black] (3.56,0.58) node {$u_{6}$};
\draw[color=black] (3.2718518518518587,0.42) node {$e_{10}$};
\draw [fill=black] (4,0.5) circle (0.6pt);
\draw[color=black] (4.06,0.58) node {$u_{7}$};
\draw[color=black] (3.7577777777777857,0.42) node {$e_{12}$};
\draw [fill=black] (4.194320987654331,0.503) circle (0.2pt);
\draw [fill=black] (4.401728395061739,0.503) circle (0.2pt);
\draw [fill=black] (4.603209876543221,0.503) circle (0.2pt);
\draw[color=black] (1.3004938271604958,0.8316872427983553) node {$e_{1}$};
\draw[color=black] (1.5493827160493854,1.0845267489711952) node {$e_{3}$};
\draw[color=black] (1.7982716049382748,1.3334156378600845) node {$e_{5}$};
\draw[color=black] (2.051111111111115,1.582304526748974) node {$e_{7}$};
\draw[color=black] (2.3,1.8311934156378633) node {$e_{9}$};
\draw[color=black] (2.560740740740746,2.0840329218107034) node {$e_{11}$};
\end{scriptsize}
\end{tikzpicture}
\end{center}
\end{example}

Next, we describe a class of graphs that always present the shadowing property. To make this precise, given a graph $E$ and an edge $e\in E^1$, we denote the immediate follower set of $e$ by $F^1(e)$. More precisely, $$F^1(e)=\{f\in E^1:s(f)=r(e)\}.$$ 

\begin{definition}\label{efifsp}
We say that a graph $E$ has the eventually constrained immediate follower set property if there exists a natural number $k$ such that if $j>k$ and $e_{\ell}\in F^1(e_j)$, then $\ell \leq k$.
\end{definition}

\begin{example}
    An example of a graph that satisfies the eventually constrained immediate follower set property is given below.

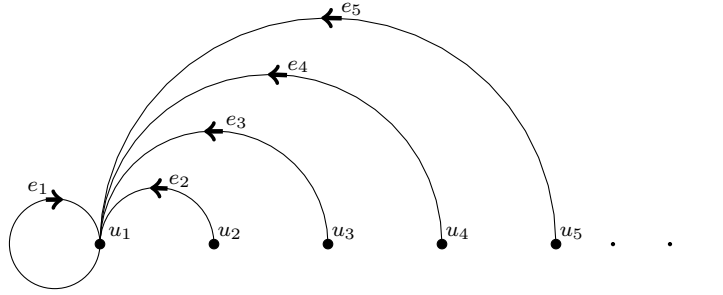
\begin{center}
\begin{tikzpicture}[scale=3][line cap=round,line join=round,>=triangle 45,x=1cm,y=1cm]
\draw [shift={(1.25,0.5)},line width=0.4pt]  plot[domain=0:3.141592653589793,variable=\t]({1*0.25*cos(\t r)+0*0.25*sin(\t r)},{0*0.25*cos(\t r)+1*0.25*sin(\t r)});
\draw [shift={(1.5,0.5)},line width=0.4pt]  plot[domain=0:3.141592653589793,variable=\t]({1*0.5*cos(\t r)+0*0.5*sin(\t r)},{0*0.5*cos(\t r)+1*0.5*sin(\t r)});
\draw [shift={(1.75,0.5)},line width=0.4pt]  plot[domain=0:3.141592653589793,variable=\t]({1*0.75*cos(\t r)+0*0.75*sin(\t r)},{0*0.75*cos(\t r)+1*0.75*sin(\t r)});
\draw [shift={(2,0.5)},line width=0.4pt]  plot[domain=0:3.141592653589793,variable=\t]({1*1*cos(\t r)+0*1*sin(\t r)},{0*1*cos(\t r)+1*1*sin(\t r)});
\draw [line width=0.4pt] (0.8013991769547331,0.5018106995884788) circle (0.19860907720257187cm);
\draw [->,line width=1.7pt] (0.7624487323304036,0.6965629227101269) -- (0.8378433692574331,0.6970474440672294);
\draw [->,line width=1.7pt] (1.2963392745601514,0.7456678074824598) -- (1.2196710584490613,0.7481534914209343);
\draw [->,line width=1.7pt] (1.8203660696199342,1.2466917812901401) -- (1.7363771106139296,1.2498762677167314);
\draw [->,line width=1.7pt] (2.0602197293389706,1.4981851452502892) -- (1.9705619590617145,1.4995666069580935);
\draw [->,line width=1.7pt] (1.5389669119069032,0.9984792671480326) -- (1.4570981884949739,0.9981560343603069);
\begin{scriptsize}
\draw [fill=black] (1.5,0.5) circle (0.6pt);
\draw[color=black] (1.57,0.55) node {$u_{2}$};
\draw [fill=black] (1,0.5) circle (0.6pt);
\draw[color=black] (1.09,0.55) node {$u_{1}$};
\draw [fill=black] (2,0.5) circle (0.6pt);
\draw[color=black] (2.07,0.55) node {$u_{3}$};
\draw [fill=black] (2.5,0.5) circle (0.6pt);
\draw[color=black] (2.57,0.55) node {$u_{4}$};
\draw [fill=black] (3,0.5) circle (0.6pt);
\draw[color=black] (3.07,0.55) node {$u_{5}$};
\draw[color=black] (1.35,0.78) node {$e_{2}$};
\draw[color=black] (1.6,1.04) node {$e_{3}$};
\draw[color=black] (1.87,1.2853497942386871) node {$e_{4}$};
\draw[color=black] (2.105,1.54) node {$e_{5}$};
\draw[color=black] (0.73,0.75) node {$e_{1}$};
\draw [fill=black] (3.25,0.5) circle (0.2pt);
\draw [fill=black] (3.5,0.5) circle (0.2pt);
\draw [fill=black] (3.75,0.5) circle (0.2pt);
\end{scriptsize}
\end{tikzpicture}
    
\end{center}
\end{example}
Graphs with the eventually constrained immediate follower set property always present the shadowing property, as we show below.

\begin{proposition}\label{propthereisk}
Let $E$ be a graph that satisfies the eventually constrained immediate follower set property and $(X,\sigma)$ be the associated Deaconu-Renault system. Then, $(X,\sigma)$ has the shadowing property.
\end{proposition}
\begin{proof}
Let $k$ be as in Definition~\ref{efifsp} and consider the following set of edges:
$$\mathcal{E} = \{e_{m}: m \leq k\} \cup \{e_{i}: \exists m \leq k \  \mbox{such that} \ s(e_{i}) = r(e_{m})  \}.$$

Given $0 < \varepsilon \leq \dfrac{1}{2^{N(k)}}$, let $\delta := \mbox{min}\left\{\varepsilon, \   \dfrac{1}{2^{N(k_{0})}}\right\},$ where $k_{0} = \mbox{max}\{\ell: e_{\ell} \in E\}$. 
First, observe that there is no finite path $\gamma$ such that $|\gamma| \geq 2$ and $\gamma_{|\gamma|} \notin F_{\delta}$. Indeed, if there is such $\gamma$, the eventually constrained immediate follower set property says that we must have $\gamma_{|\gamma|-1} \in \{e_{m}: m \leq k\}$. Then, simultaneously, whe have that $\gamma_{|\gamma|} \in \{e_{i}: \exists m \leq k \  \mbox{s. t.} \ s(e_{i}) = r(e_{m})  \}$ and $\gamma_{|\gamma|} \notin F_{\delta}$, which contradicts the inequality $\delta \leq \frac{1}{2^{N(k_{0})}}$. 

To prove shadowing, we show that $(X,\sigma)$ satisfies the First and the Second Infinite Path Conditions (Definitions \ref{FIPC} and \ref{SIPC}). By the observation in the paragraph above, there is no infinite set $\{\lambda^{1}, \lambda^{2}, \ldots\}$ of finite paths which satisfy Conditions~1. and 2 of Definition~\ref{FIPC}, and hence $(X,\sigma)$ satisfies the First Infinite Path Condition. Again by the observation in the paragraph above, there is no finite set $\{\lambda^{1}, \lambda^{2}, \ldots, \lambda^{\ell}\}$ of finite paths which satisfy Conditions~1. and 2 of Definition~\ref{SIPC}. Hence, we just need to consider an infinite path $\gamma$ such that $\gamma_{1} \notin F_{\delta}$. In this case, by choosing $\lambda := \gamma$ we have the Second Infinite Path Condition satisfied. Therefore, by Theorem~\ref{teoinfiniteshadowing}, $(X,\sigma)$ has the infinite shadowing property.
\end{proof}

With the above result, we can now completely characterize the shadowing property for graphs that contain an attractor subgraph.

\begin{proposition}\label{propproAshadowing}
Let $E$ be an infinite graph that contains a finite attractor subgraph and $(X,\sigma)$ the associated Deaconu-Renault system. Then $(X,\sigma)$ has the shadowing property if, and only if, $E$ satisfies the eventually constrained immediate follower set property.
\end{proposition}
\begin{proof}
Let $E'$ be a finite attractor subgraph of  $E$. By Proposition~\ref{propthereisk}, we have that the eventually constrained immediate follower set property implies the shadowing property of $X$. We prove the other implication. 

Suppose that $E$ does not satisfy the eventually constrained immediate follower set property. This means that for all natural $k$ there are edges $e_{j_{k}}$ and $e_{\ell_{k}}$ such that $r(e_{j_{k}}) = s(e_{\ell_{k}})$ and $j_{k},\ell_{k} > k$. Let $m' = \mbox{max}\{m: e_{m} \in \mathcal{G'}^{1}\}$, fix $\varepsilon = \dfrac{1}{2^{N(m')}}$, and let $\delta$ be any real number less than $\varepsilon$. Let $k_{0}$ be a natural number such that if $k > k_{0}$, and $e_{j_{k}} = p_{r(k)}$ and $e_{\ell_{k}} = p_{s(k)}$ (in the enumeration of $\p$), then $\frac{1}{2^{r(k)}},\frac{1}{2^{s(k)}}<\delta$. Finally, for each natural number $n \geq 1$, consider the finite path $\lambda^{n}:=e_{j_{k_{0}+n}}e_{\ell_{k_{0}+n}}$. Since $X$ does not admit an infinite path $\lambda$ such that $\lambda_{j} \notin F_{\varepsilon}$ for all $j \geq 1$ (because $E$ has a finite attractor subgraph), we have that $(X,\sigma)$ does not satisfy the First Infinite Path Condition (Definition~\ref{FIPC}) and hence $(X,\sigma)$ does not present the shadowing property (see Theorem~\ref{teoinfiniteshadowing}). 

\end{proof}

\begin{corollary}
The renewal shift, of Example~\ref{renewal}, does not have the shadowing property.
\end{corollary}
\begin{proof}
    As we saw in Example~\ref{renewal}, the renewal shift has an attractor subgraph. Clearly, it does not satisfy the eventually constrained immediate follower set property. The result now follows from the proposition above.
\end{proof}

\subsection{A graph with two vertices without the finite shadowing property.}

In this section, we give an example of an infinite graph with only two vertices but that does not present the finite shadowing property. Precisely, let $E_2$ be the graph with $E_2^0=\{u,v\}$, $E_2^1=\{e_{k}: k \geq 1\}$, and such that $s(e_{1}) = r(e_{2}) = u$, $r(e_{1}) = s(e_{2}) = v$ and, for all natural $k \geq 2$, $s(e_{2k-1}) = r(e_{2k-1}) = u $ and $s(e_{2k}) = r(e_{2k}) = v$. A picture of this graph is shown below.

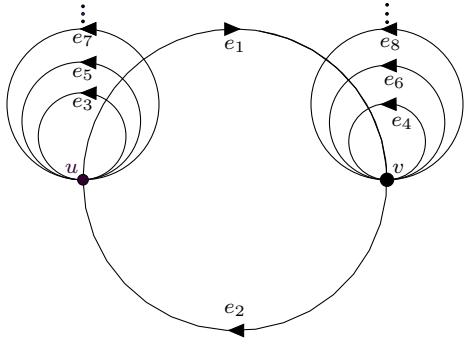
\begin{center}
\definecolor{ududff}{rgb}{0.30196078431372547,0.30196078431372547,1}
\definecolor{ttqqtt}{rgb}{0.2,0,0.2}
\begin{tikzpicture}[line cap=round,line join=round,>=triangle 45,x=1cm,y=1cm]
\draw [line width=0.4pt] (-6,2) circle (1cm);
\draw [line width=0.4pt] (-6.0225,1.78) circle (0.7803244517506808cm);
\draw [line width=0.4pt] (-6.0114069828035435,1.5765346534653466) circle (0.5766474884217285cm);
\draw [->,line width=0.4pt] (-5.961781599302772,2.1510428320106905) -- (-6.039620787712179,2.1524915156656896);
\draw [->,line width=0.4pt] (-5.958366678740344,2.5576844907185734) -- (-6.042407920026582,2.560070461381673);
\draw [->,line width=0.4pt] (-5.966902636245555,2.999452132176677) -- (-6.043190021331042,2.999066875668203);
\draw [->,line width=0.4pt] (-5.966902636245555,2.999452132176677) -- (-6.043190021331042,2.999066875668203);
\draw [line width=0.4pt] (-2,2) circle (1cm);
\draw [line width=0.4pt] (-1.9984110595295315,1.504176329250862) circle (0.5041788330629245cm);
\draw [line width=0.4pt] (-1.9984110595295321,1.7587126775449888) circle (0.7587143413692039cm);
\draw [shift={(-4,1)},line width=0.4pt]  plot[domain=0:3.141592653589793,variable=\t]({1*2*cos(\t r)+0*2*sin(\t r)},{0*2*cos(\t r)+1*2*sin(\t r)});
\draw [->,line width=.4pt] (-4.0832903979140625,2.9982649247823265) -- (-3.9222560335781442,2.9984883976858603);
\draw [shift={(-4,1)},line width=0.4pt]  plot[domain=-3.141592653589793:1.4082852609269758,variable=\t]({1*2*cos(\t r)+0*2*sin(\t r)},{0*2*cos(\t r)+1*2*sin(\t r)});
\draw [->,line width=0.4pt] (-3.9177793338884555,-0.9983092258367257) -- (-4.098118047063992,-0.9975917623078916);
\draw [->,line width=0.4pt] (-1.9205060262860576,2.996835346555873) -- (-2.0860248790077662,2.9962929891310583);
\draw [->,line width=.4pt] (-1.925328511468263,2.5138990075114265) -- (-2.090723792668377,2.5117902373618652);
\draw [->,line width=.4pt] (-1.9195248593346776,2.002145467980874) -- (-2.092521406119482,1.999493942882168);
\begin{scriptsize}
\draw[color=black] (-6,2.85) node {$e_{7}$};
\draw [fill=ttqqtt] (-6,1) circle (2pt);
\draw[color=ttqqtt] (-6.15,1.1545967611610057) node {$u$};
\draw[color=black] (-6,2.4) node {$e_{5}$};
\draw[color=black] (-6,2) node {$e_{3}$};
\draw [fill=ududff] (-6,3.2) circle (0.5pt);
\draw [fill=ttqqtt] (-6,3.3) circle (0.5pt);
\draw [fill=ttqqtt] (-6,3.1) circle (0.5pt);
\draw [fill=black] (-2,1) circle (2.5pt);
\draw[color=black] (-1.85,1.1545967611610057) node {$v$};
\draw[color=black] (-1.95,2.8) node {$e_{8}$};
\draw[color=black] (-1.8,1.75) node {$e_{4}$};
\draw[color=black] (-1.9,2.3) node {$e_{6}$};
\draw[color=black] (-4,2.77) node {$e_{1}$};
\draw[color=black] (-4,-0.72) node {$e_{2}$};
\draw [fill=black] (-1.9984110595295315,3.0901335763142677) circle (0.5pt);
\draw [fill=black] (-2,3.2) circle (0.5pt);
\draw [fill=black] (-1.9984110595295315,3.3201952757339592) circle (0.5pt);
\end{scriptsize}
\end{tikzpicture}
\end{center}

\begin{proposition}
    The Deaconu-Renault system $(X,\sigma)$, associated with the graph $E_2$ defined above, does not present the finite shadowing property.
\end{proposition}

\begin{proof}

Recall that we are using an edge-ordered enumeration. Let $\varepsilon = \frac{1}{2^{j}}$, where  $e_{4} = p_{j}$ and consider any $\delta < \varepsilon$. Let $n > 3$ be a natural number such that $e_{n} = p_{\ell}$, with $\frac{1}{2^{\ell}} < \delta$. 
We show that $(X,\sigma)$ does not satisfy Finite Path Condition (Definition \ref{FPC1}).

Consider the finite set of finite paths $\{\lambda^{1}, \lambda^{2}\}$, where $\lambda^{1} := e_{3}e_{2n-1}$ and $\lambda^{2}:= e_{2n}e_{4}$, and notice that the space $X$ associated with $E$ does not admit a finite path of length 3, say $\lambda$, which begins with  $e_{3}$, ends with $e_{4}$, and such that $\lambda_{2} \notin F_{\varepsilon}$. Therefore, $(X,\sigma)$ does not satisfy the Finite Path Condition and hence it does not present the finite shadowing property (see Theorem \ref{teofiniteshadowing}). 
\end{proof}

\subsection{A graph with the finite shadowing property and without the infinite shadowing property}

In this section, we show that the finite shadowing property does not coincide with the shadowing property. For this, we consider the following graph.

Let $E_{f}$ be the graph such that $E_f^0=\{v_{k}: k \geq 1\}$, $E_f^1=\{e_{k}: k \geq 1\}$, and the range and source maps are given by:
\begin{eqnarray*}
s(e_{1}) &=& r(e_{1}) = r(e_{2})= v_{1};\\
s(e_{k}) &=& v_{k}  =  r(e_{k+1})  \ \ \mbox{for  $k \geq 2$.}
\end{eqnarray*}
A picture of this graph is shown below.
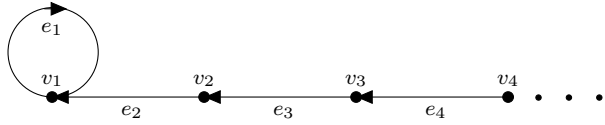
\vspace{0,3 cm}
\begin{center}
\begin{tikzpicture}[line cap=round,line join=round,>=triangle 45,x=1cm,y=1cm]
\draw [->,line width=0.4pt] (8,0) -- (6,0);
\draw [->,line width=0.4pt] (6,0) -- (4,0);
\draw [->,line width=0.4pt] (4,0) -- (2,0);
\draw [line width=0.4pt] (2.0133333333333336,0.5933333333333333) circle (0.5934831271588286cm);
\draw [->,line width=0.4pt] (1.9,1.17) -- (2.2,1.17);
\begin{scriptsize}
\draw [fill=black] (8,0) circle (2pt);
\draw[color=black] (8,0.22) node {$v_{4}$};
\draw [fill=black] (6,0) circle (2pt);
\draw[color=black] (6,0.22) node {$v_{3}$};
\draw[color=black] (7.05,-0.2) node {$e_{4}$};
\draw [fill=black] (4,0) circle (2pt);
\draw[color=black] (4,0.22) node {$v_{2}$};
\draw[color=black] (5.05,-0.2) node {$e_{3}$};
\draw [fill=black] (2,0) circle (2pt);
\draw[color=black] (2,0.22) node {$v_{1}$};
\draw[color=black] (3.05,-0.2) node {$e_{2}$};
\draw[color=black] (2,0.9) node {$e_{1}$};
\draw [fill=black] (8.4,-0.006666666666666665) circle (1pt);
\draw [fill=black] (8.8,-0.006666666666666665) circle (1pt);
\draw [fill=black] (9.2,-0.006666666666666665) circle (1pt);
\end{scriptsize}
\end{tikzpicture}    
\end{center}

\begin{proposition}\label{exempfinitevinfinite}
Let $E_{f}$ be the graph defined above and $(X,\sigma)$ the associated Deaconu-Renault system. Then, $(X,\sigma)$ has the finite shadowing property but does not have the infinite shadowing property.
\end{proposition}
\begin{proof}

We prove first that $(X,\sigma)$ presents the finite shadowing property.
Given $\varepsilon > 0$, let $0 < \delta < \varepsilon$.
Notice that if $\gamma$ is a finite path such that $|\gamma| \geq 2$ and $\gamma_{1},\gamma_{|\gamma|} \notin F_{\delta}$, then $\gamma_{j} \notin F_{\varepsilon}$. Let $\{\lambda^{1}, \lambda^{2}, \ldots, \lambda^{\ell}\}$ be a finite set of finite paths satisfying conditions 1. and 2. of Definition \ref{FPC1}. Choose $\lambda$ as any finite path such that $\displaystyle |\lambda| = \left(\sum_{i=1}^{\ell}|\lambda^{i}|\right)-\ell + 1$ and $\lambda_{f({\ell},1)} =  \lambda^{\ell}_{1}$. Then, the Finite Path Condition is satisfied, and hence $(X,\sigma)$ satisfies the finite shadowing property (Theorem~\ref{teofiniteshadowing}).

To see that $(X,\sigma)$ does not present the infinite shadowing property, notice that $\{v_{1},e_{1}\}$ is a finite attractor subgraph of $E_f$ (Definition~\ref{attractor}) and $E_{f}$ does not satisfy the eventually constrained immediate follower set property  (Definition~\ref{efifsp}). If follows from Proposition~\ref{propproAshadowing} that $(X,\sigma)$ does not present the infinite shadowing property.
\end{proof}

\begin{remark}
As a final remark, we notice that slight variations on the graph in Example~\ref{exempfinitevinfinite} lead us to changes concerning shadowing, which illustrates the complexity of shadowing when dealing with infinite graphs and their associated Deaconu-Renault systems. For instance, consider the graph $E$ whose vertices and edges are represented, respectively, by $\{v_{k}: k \geq 1\}$ and $\{e_{k}: k \geq 1\}$ and
\begin{eqnarray*}
s(e_{1}) &=& v_{1}; \\
s(e_{k+1}) &=& r(e_{k})  =  v_{k+1}  \ \ \mbox{for  $k \geq 1$.}
\end{eqnarray*}  

Notice that this graph (which is the same as the one in Example~\ref{examplewandering}) is wandering (Definition \ref{defiwandering}) and hence the associated Deaconu-Renault system has the infinite shadowing property (Proposition~\ref{propwandering}).

Now, consider $E_{1}$ the graph whose vertices and edges are represented, respectively, by $\{v_{k}: k \geq 1\}$ and $\{e_{k}: k \geq 1\}$, and the range and source maps are given by:
\begin{eqnarray*}
s(e_{2k-1}) &=& v_{2k-1}; \ \ \mbox{for $k \geq 1$;} \\
r(e_{1}) &=& v_{2};  \\
r(e_{2k-1}) &=& v_{2k+1} \ \ \mbox{for $k \geq 2$;}\\
s(e_{2k}) &=& v_{2k} \ \ \mbox{for $k \geq 1$;} \\
r(e_{2k}) &=& v_{2(k+1)} \ \ \mbox{for $k \geq 1$.}
\end{eqnarray*}
A picture of this graph is shown below.

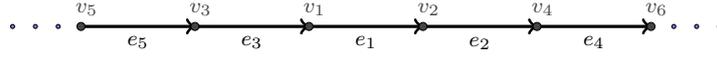
\definecolor{xdxdff}{rgb}{0.49019607843137253,0.49019607843137253,1}
\definecolor{uuuuuu}{rgb}{0.26666666666666666,0.26666666666666666,0.26666666666666666}
\begin{center}
\begin{tikzpicture}[scale=1.5][line cap=round,line join=round,>=triangle 45,x=1cm,y=1cm]
\draw [->,line width=1.2pt] (0,0) -- (1,0);
\draw [->,line width=1.2pt] (1,0) -- (2,0);
\draw [->,line width=1.2pt] (2,0) -- (3,0);
\draw [->,line width=1.2pt] (3,0) -- (4,0);
\draw [->,line width=1.2pt] (4,0) -- (5,0);
\begin{scriptsize}
\draw [fill=uuuuuu] (0,0) circle (1.0pt);
\draw[color=uuuuuu] (0.05,0.15) node {$v_{5}$};
\draw [fill=uuuuuu] (1,0) circle (1.0pt);
\draw[color=uuuuuu] (1.05,0.15) node {$v_{3}$};
\draw[color=black] (0.5,-0.14) node {$e_{5}$};
\draw [fill=uuuuuu] (2,0) circle (1.0pt);
\draw[color=uuuuuu] (2.05,0.15) node {$v_{1}$};
\draw[color=black] (1.5,-0.14) node {$e_{3}$};
\draw [fill=uuuuuu] (3,0) circle (1.0pt);
\draw[color=uuuuuu] (3.05,0.15) node {$v_{2}$};
\draw[color=black] (2.5,-0.14) node {$e_{1}$};
\draw [fill=uuuuuu] (4,0) circle (1.0pt);
\draw[color=uuuuuu] (4.05,0.15) node {$v_{4}$};
\draw[color=black] (3.5,-0.15) node {$e_{2}$};
\draw [fill=uuuuuu] (5,0) circle (1.0pt);
\draw[color=uuuuuu] (5.05,0.15) node {$v_{6}$};
\draw[color=black] (4.5,-0.14) node {$e_{4}$};
\draw [fill=xdxdff] (5.2,0) circle (0.5pt);
\draw [fill=xdxdff] (5.4,0) circle (0.5pt);
\draw [fill=xdxdff] (5.6,0) circle (0.5pt);
\draw [fill=xdxdff] (-0.2,0) circle (0.5pt);
\draw [fill=xdxdff] (-0.4,0) circle (0.5pt);
\draw [fill=xdxdff] (-0.6,0) circle (0.5pt);
\end{scriptsize}
\end{tikzpicture}
    
\end{center}

We show that the associated Deaconu-Renault system $(X,\sigma)$ does not present the finite shadowing property. Choose $\varepsilon = \frac{1}{2^{j}}$, where $e_{1} = p_{j}$, and let $\delta < \varepsilon$. Define $2k$ as the smallest even natural number such that $e_{2k} = p_{j'}$ and $\frac{1}{2^{j'}} < \delta$, and consider the finite set of finite paths $\{\lambda^{1}, \lambda^{2}\}$, where  $\lambda^{1}_{1} = e_{1}$, $\lambda^{1}_{|\lambda^{1}|} = e_{2k}$; $\lambda^{2}_{1} = e_{2k+1}$,  $\lambda^{2}_{|\lambda^{2}|} = e_{1}$. Notice that $\{\lambda^{1}, \lambda^{2}\}$ satisfies Conditions~1. and 2. of Definition \ref{FPC1}, but, as the graph does not admit any finite path which begins and ends at $e_{1}$, we conclude that $(X,\sigma)$ does not satisfy the Finite Path Condition and, therefore, does not present the finite shadowing property Theorem~\ref{teofiniteshadowing}). 
\end{remark}

\vspace{5mm}
Daniel Gon\c{c}alves, 
 { Departamento de Matem\'atica, Universidade Federal de Santa Catarina, 88040-900, Florian\'opolis SC, Brazil.}
 
{\textit{Email Address: }}\texttt{{daemig@gmail.com}}

\vspace{3mm}

Bruno Brogni Uggioni,
{Instituto Federal do Rio Grande do Sul, 92412-240, Canoas RS, Brazil.}

{\textit{Email Address: }}\texttt{{brunobrogni@gmail.com}}

\end{document}